\RequirePackage{fix-cm}
\documentclass{article}                     % onecolumn (standard format)
\usepackage{graphicx}
\usepackage{amsmath,amssymb,amsfonts,amsthm}
\newtheorem{theorem}{Theorem}[section]

\newtheorem{remark}{Remark}[section]
\newtheorem{lemma}[theorem]{Lemma}
\newtheorem{defi}{Definition}[section]

\newtheorem{cor}{Corollary}[section]

\begin{document}

\title{A BGK model for gas mixtures of polyatomic molecules allowing for slow and fast relaxation of the temperatures %\thanks{Grants or other notes
}


\author{Marlies Pirner    %     \and
        %Second Author %etc.
}


%\institute{Marlies Pirner \at
          %    University of W\"urzburg, Emil-Fischer-Str. 40,\\ 97074 W\"urzburg, Germany \\
           %   Tel.: +49 (0)931 31-86154\\
           %   Fax: +49 (0)931 31-83494\\
%              \email{marlies.pirner@mathematik.uni-wuerzburg.de}           %  \\
    %       \and
    %       S. Author \at
    %          second address
%}

\date{}
\maketitle

\begin{abstract}
Kinetic models for polyatomic gases have two temperatures for the two different types of degrees of freedom, the translational and the internal energy degrees of freedom. Therefore, in the case of BGK models one expects two types of relaxations, a relaxation of the distribution function to a Maxwell distribution and a relaxation of the two temperatures to an equal value. The speed for the first type of relaxation may be faster or slower than the second type of relaxation. Models found in the literature  often allow  only for one of these two cases. We believe that a model should allow for both cases. That is why we derive a new multi-species polyatomic BGK model which allows for both regimes.
For this new model we prove conservation properties,  positivity of the temperatures, the H-theorem and characterize the equilibrium as a Maxwell distribution with equal temperatures. Moreover, we prove the convergence rate to equilibrium and that we can actually capture both regimes of the relaxation processes.

\textbf{keywords:} kinetic model, BGK model, polyatomic molecules, converge rate to equilibrium, entropy dissipation estimates  \\

\textbf{PACS subject classification:} 02.30.Jr , 05.20.Dd, 33.20.Vq\\

\end{abstract}

\section{Introduction}
 In this paper we shall concern ourselves with a kinetic description of gas mixtures for polyatomic molecules. In the case of mono atomic molecules and two species this is traditionally done via the Boltzmann equation for the density distributions $f_1$ and $f_2$, see for example \cite{Cercignani,Cercignani_1975}. Under certain assumptions the complicated interaction terms of the Boltzmann equation can be simplified by a so called BGK approximation, consisting of a collision frequency multiplied by the deviation of the distributions from local Maxwellians. This approximation should be constructed in a way such that it  has the same main properties of the Boltzmann equation namely conservation of mass, momentum and energy, further it should have an H-theorem with its entropy inequality and the equilibrium must still be Maxwellian.  BGK  models give rise to efficient numerical computations, which are asymptotic preserving, that is they remain efficient even approaching the hydrodynamic regime \cite{Puppo_2007,Jin_2010,Dimarco_2014,Bennoune_2008,Bernard_2015,Crestetto_2012}. Evolution of a polyatomic gas is very important in applications, for instance air consists of a gas mixture of polyatomic molecules. But, most kinetic models modelling air deal with the case of a mono atomic  gas consisting of only one species.
% 
%  In the literature one can find two types of models for polyatomic molecules. There are models which contain a sum of collision terms on the right-hand side corresponding to the elastic and inelastic collisions. Examples are the models of Rykov \cite{Rykov}, Holway \cite{Holway} and Morse \cite{Morse}. The other type of models contain only one collision term on the right-hand side taking into account both elastic and inelastic interactions. Examples for this are Bernard, Iollo, Puppo \cite{Bernard}  or the model by Bisi and Caceres \cite{Bisi} modelling chemical interactions. In this paper we want to extend the model of Bernard, Iollo and Puppo \cite{Bernard} from one species of molecules to a gas mixture of polyatomic molecules. 

In contrast to mono atomic molecules, in a polyatomic gas energy is not entirely stored in the kinetic energy of its molecules but also in their rotational and vibrational modes. In addition, one has  two temperatures. One temperature is related to the translational degrees of freedom, the other temperature is related to the degrees of freedom in internal energy. Due to the principle of equipartition of energy, one expects from physics that in equilibrium the two temperatures coincides. Therefore, in the case of polyatomic molecules, one expects an additional relaxation of the two temperatures to a common value in addition to the relaxation of the distribution function to a Maxwell distribution. In this paper, we want to analyse the principles of realizing this in the case of different existing BGK models for polyatomic molecules in the literature, both for one species and for gas mixtures. In the case of one species, the first one is the model by Andries, Le Tallec, Perlat and Perthame \cite{Perthame}, the second one is the model of Klingenberg, Pirner and Puppo \cite{Pirner5} reduced to one species and the last one, the model of Bernard, Iollo, Puppo in \cite{Bernard}. In the case of gas mixture, we analyse the model presented in \cite{Pirner5}. For these models, we want to understand the principle of the additional relaxation of the temperatures, the physical regimes where it is reasonable to use them, and prove  the convergence rate to equilibrium in the space-homogeneous case in order to see if it reflects a reasonable physical behaviour.

We prove the convergence rate by considering the entropy dissipation estimates for the relative entropy. Then, the relative entropy can be related to the $L^1$-norm by using the Ciszar- Kullback inequality. This method is introduced for example in \cite{Matthes}, and is also used for example in \cite{BisiCaizoLods} for the linear Boltzmann operator, in \cite{YunES} for the ES-BGK model and in \cite{Yunpoly} for the model in \cite{Perthame}.

We will observe that two of the three models for one species and the model for gas mixtures have a restricted physical validity. Therefore, we will also present a new model for the case of gas mixtures which has a more general physical validity, and in addition, we prove that the convergence rate to equilibrium reflects a reasonable physical behaviour.

% For simplification we present the model in the case of two species. We allow the two species to have different degrees of freedom in internal energy. For example, we may consider a mixture consisting of a mono atomic and a diatomic gas. In addition, we want to model it via an ES-BGK approach  in order to reproduce the correct Boltzmann hydrodynamic regime close to the asymptotic continuum limit. The ES-BGK approximation was suggested  by Holway in the case of one species \cite{Holway}. The H-Theorem of this model then was proven in \cite{Perthame}. For polyatomic molecules this was done before by Brull and Schneider \cite{brullschneider} for one species.

The outline of the paper is as follows: in section \ref{sec1} we will present three existing models for one species \cite{Perthame}, \cite{Pirner5} and \cite{Bernard} in the literature, prove the convergence rates to equilibrium in the space homogeneous case and discuss the physical validity of the three models. In section \ref{sec2.2} we will present an existing model for gas mixtures \cite{Pirner5} in the literature, prove the convergence rate to equilibrium in the space homogeneous case and discuss the physical validity. In section \ref{sec2}, we present a new model for gas mixtures of polyatomic molecules which has a larger physical validity and prove the H-theorem and the convergence rate to equilibrium for this new model.
%In section \ref{sec8} we apply the method of Chu reduction to our model in order to reduce the complexity of the variables for the rotational and vibrational energy degrees of freedom for numerical purposes.
%In section \ref{sec9} we give an application in the case of a mono atomic and a diatomic molecule.

%\newpage
\section{Comparison of entropy dissipation estimates for BGK models for one species of polyatomic molecules from the literature}
\label{sec1}
In a polyatomic gas there are two types of relaxations in order to reach equilibrium. In a polyatomic gas, one expects that the distribution function relaxes towards a Maxwell distribution. In addition, in a polyatomic gas, one has two different temperatures, one temperature which is related to the translational degrees of freedom and another temperature which is related to the degrees of freedom in internal energy. In equilibrium, one expects that the two temperatures are the same due to equipartition of energy. So one expects that the two temperatures relax towards an equal value in addition to the relaxation of the distribution functions to Maxwell distributions. Note that, there can be two different cases. The speed of relaxation of the two temperatures to an equal value can be slower or faster than the speed of relaxation to a Maxwell distribution. In the following  subsections, we analyse three models in the literature for one species of polyatomic molecules whether they cover both cases meaning if it is possible to describe both a faster and a slower relaxation of the temperatures. In addition, we prove the convergence rate to equilibrium in the space homogeneous case for each model, and check, if the rate of convergence always coincides with the slowest rate of relaxation.
\subsection{The BGK model of Andries, Le Tallec, Perlat and Perthame for one species of  polyatomic molecules}
\label{sec2.1}
We will now consider the BGK model for a single species  of polyatomic molecules presented in \cite{Perthame}. In \cite{Perthame}, they actually present an ES-BGK model. But in order to understand the effects leading to the relaxation of the distribution function to a Maxwell distribution and the relaxation of the temperatures to an equal value,  we want to look at the corresponding simpler BGK model.

In \cite{Perthame}, they consider a distribution function $f(x,v,I,t)$ depending on the position $x\in \mathbb{R}^3$, the velocity $v\in\mathbb{R}^3$ and internal energy $\varepsilon(I)= I^{\frac{2}{\delta}}$, $I\in\mathbb{R}^+$ at time $t>0$. The coefficient $\delta$ denotes the number of degrees of freedom in internal energy. In \cite{Perthame}, it is assumed that the mass of the particles is equal to $1$. In the following, we assume additionally  that $k_{B}=1$ in this model. The density $\rho$ and mean velocity $u$ are defined as 
\begin{align*}
\rho(x,t) = \int f(x,v, I, t) dv dI, ~ \rho(x,t) u(x,t) = \int f(x,v,I,t) v dv dI.
\end{align*}
 The energy is defined as 
$$ E(x,t) = \int \int \left(\frac{1}{2} |v|^2 + I^{\frac{\delta}{2}}\right) f(x,v,I,t) dv dI = \frac{1}{2} \rho(x,t) |u(x,t)|^2 + \rho(x,t) e(x,t).$$
The specific internal energy can be divided into 
$$ e_{tr} = \frac{1}{\rho} \int \int \frac{1}{2} |v-u|^2 f dv dI,$$
$$e_{int} = \frac{1}{\rho} \int \int  I^{\frac{2}{\delta}} f dv dI,$$
and associate with this the corresponding temperatures
$$e= e_{tr} + e_{int}= \frac{3+ \delta}{2} T_{equ},$$
$$ e_{tr}=\frac{3}{2}  T_{tr},$$
$$e_{int} = \frac{\delta}{2}  T_{int},$$
and define $T_{rel} = \theta T_{equ} + (1- \theta) T_{int},$ with $0< \theta \leq 1$. In \cite{Perthame}  the following  Gaussian for the single species BGK model is considered
$$ \widetilde{G}[f]= \frac{\rho \Lambda_{\delta}}{\sqrt{2 \pi T}^3} \frac{1}{ T_{rel}^{\frac{\delta}{2}}} \exp \left( - \frac{1}{2} \frac{|v-u|^2}{ T} - \frac{I^{\frac{\delta}{2}}}{ T_{rel}} \right),$$ 
with the temperature $T = (1- \theta)  T_{tr}  + \theta  T_{equ} $.  $\Lambda_{\delta}$ is a constant ensuring that the integral of $\widetilde{G}[f]$ with respect to $v$ and $I$ is equal to the density $\rho$. Then the model is given by
$$ \partial_t f + v \cdot \nabla_x f = A_{\nu} (\widetilde{G}[f] - f)$$
with the collision frequency $A_{\nu}$.

For this model one can prove conservation of the number of particles, momentum and total energy, and an H-theorem such that the equilibrium is characterized by a Maxwell distribution with equal temperatures $T_{equ}= T_{tr} = T_{int}$, for details see section 3 in \cite{Perthame}. It is also ensured that there exists a unique mild solution to this model. This is proven in \cite{Yunexpoly}.

The convex combination in $\theta$ takes into account that $T_{tr}$ and $T_{int}$ relax towards the common value $T_{equ}$. In the space-homogeneous case we see that we get the following macroscopic equations
\begin{align*}
\partial_t T_{tr} &= A_{\nu}  ( T_{tr} (1- \theta) + \theta T_{equ} - T_{tr})= A_{\nu} \theta (T_{equ} - T_{tr}),
\\
\partial_t T_{int} &= A_{\nu} \theta (T_{equ} - T_{int}) ,
\end{align*}
 These macroscopic equations describe a relaxation of $T_{tr}$ and $T_{int}$ towards $T_{equ}$. We see that the model captures only the regime where this relaxation is slower than the relaxation of the distribution function to a Maxwell distribution since $\theta$ satisfies $\theta \leq 1$, so it reduces the speed of relaxation from $A_{\nu}$ to $A_{\nu}  \theta$. 

This model satiesfies the following assymptotic behaviour proven in \cite{Yunpoly} in the space-homogeneous case.

\begin{theorem}
Let $0<\theta \leq 1.$ The distribution function for the spatially homogeneous case converges to equilibrium with the following rate:
$$||f(t) - M_{0,1}||_{L^1(dv dI)} \leq e^{- \frac{\theta}{2} A_{\nu} t }\sqrt{2 H(f_0 |M_{0,1})},$$
with the relative entropy $H(f|g)= \int \int f \ln \frac{f}{g} dv dI$ for two functions $f$ and $g$, and the Maxwell distribution $M_{0,1}$ given by
$$M_{0,1} = \frac{\rho \Lambda_{\delta}}{(2\pi T_{equ})^{3/2} (T_{equ})^{\delta/2}} e^{- \frac{|v-u|^2}{2 T_{equ}} - \frac{I^{2/\delta}}{T_{equ}}}.$$
\end{theorem}
We observe that the property that the speed of relaxation of the temperatures $T_{tr}$ and $T_{int}$ relax to the same temperature $T_{equ}$ is slower than the speed of relaxation to Maxwell distributions, is reflected in the convergence rate to equilibrium. The rate of convergence depends on the factor $\theta A_{\nu}$ with $0<\theta \leq 1.$
\subsection{The BGK model of Klingenberg, Pirner and Puppo for one species of polyatomic molecules}
\label{2.2}
In this section we will present the model from the literature for one species described in \cite{Pirner5}. For the convenience of the reader, we briefly repeat it here. For more details and motivation for the choice of the model see \cite{Pirner5}. %In \cite{Pirner5} the model is presented for two species. Since gas mixtures are important in applications, we will present it here also in the case of two species. Then, in addition to a comparison with the model from the previous subsection, we can also analyse for the two species model in which way it relaxes to equilibrium.

 Let $x\in \mathbb{R}^d$ and $v\in \mathbb{R}^d, d \in \mathbb{N}$ be the phase space variables  and $t\geq 0$ the time. Let $l$ be the number of internal degrees of freedom. Further, $\eta \in \mathbb{R}^{l}$  is the variable for the internal energy degrees of freedom.
  %Let $M$ be the total number of different rotational and vibrational degrees of freedom and $l_k$ the number of internal degrees of freedom of species $k$, $k=1,2$. Note that the sum $l_1+l_2$ is not necessarily equal to $M$, because $M$ counts only the different degrees of freedom in the internal energy, $l_1+l_2$ counts all degrees of freedom in the internal energy. For example, consider two species consisting of diatomic molecules which have two rotational  degrees of freedom. In addition, the second species has one vibrational degree of freedom. Then we have $M=3, l_1=2, l_2=3$. Further, $\eta \in \mathbb{R}^{M}$  is the variable for the internal energy degrees of freedom, $\eta_{l_k} \in \mathbb{R}^{M}$ coincides with $\eta$ in the components corresponding to the internal degrees of freedom of species $k$ and is zero in the other components.  \\ Since we want to describe two different species, our kinetic model has two distribution functions $f_1(x,v,\eta_{l_1},t)> 0$ and $f_2(x,v,\eta_{l_2},t) > 0$. 
  Then the kinetic model has one distribution function $f(x,v, \eta,t) <0$.
 Furthermore, for any $f: \Lambda_{poly} \times \mathbb{R}^d \times \mathbb{R}^l \times \mathbb{R}^+_0, \Lambda_{poly} \subset \mathbb{R}^d$ with $(1+|v|^2 + |\eta|^2)f \in L^1(dv d\eta)$, $f \geq 0,$  we relate the distribution functions to  macroscopic quantities by mean-values of $f$ as follows
\begin{align}
\int f(v, \eta) \begin{pmatrix}
1 \\ v \\ \eta \\ m |v-u|^2 \\ m |\eta - \bar{\eta} |^2 \\ m (v-u(x,t)) \otimes (v-u(x,t))
\end{pmatrix} 
dv d\eta=: \begin{pmatrix}
n \\ n u \\ n \bar{\eta} \\ d n T^{t} \\ l n T^{r} \\ \mathbb{P}
\end{pmatrix} , 
\label{momentsone}
\end{align} 
 where $n$ is the number density, $u$ the mean velocity, $\bar{\eta}$ the mean variable related to the internal energy, $T^{t}$ the mean temperature of the translation, $T^{r}$ the mean temperature of the internal energy degrees of freedom for example rotation or vibration and $\mathbb{P}$ the pressure tensor. Note that in this paper we shall write $T^{t}$ and $T^{r}$ instead of $k_B T^{t}$ and $k_B T^{r}$, where $k_B$ is Boltzmann's constant. 
In the following, we always keep the term $\bar{\eta}$ in order to cover the most general case, but in  \cite{Pirner5}, they require $\bar{\eta}=0$, which means that the energy in rotations clockwise is the same as in rotations counter clockwise. Similar for vibrations, whereas in \cite{Pirner8}, it is shown that if one requires $\bar{\eta} = \omega$ with a fixed $\omega \in \mathbb{R}^l$ such that $|\omega|^2 = 2 \frac{p_{\infty}}{m n}$, this leads to a more general equation of state in equilibrium given by $p= n T + const.$

We consider the model presented in \cite{Pirner5}  given by

\begin{align} \begin{split} \label{BGKone}
\partial_t f + v\cdot\nabla_x   f   &= \nu n (M - f) %+ \nu_{12} n_2 (M_{12}- f_1),
%\\ 
%\partial_t f_2 + v\cdot\nabla_x   f_2 &=\nu_{22} n_2 (M_2 - f_2) + \nu_{21} n_1 (M_{21}- f_2), \\
%f_1(t=0) &= f_1^0, \\
%f_2(t=0) &= f_2^0
\end{split}
\end{align}
with the Maxwell distribution
\begin{align} 
\begin{split}
M(x,v,\eta,t) &= \frac{n}{\sqrt{2 \pi \frac{\Lambda}{m}}^d } \frac{1}{\sqrt{2 \pi \frac{\Theta}{m}}^{l}} \exp({- \frac{|v-u|^2}{2 \frac{\Lambda}{m}}}- \frac{|\eta- \bar{\eta}|^2}{2 \frac{\Theta}{m}}), 
%\\
%M_2(x,v,\eta_{l_2},t) &= \frac{n_2}{\sqrt{2 \pi \frac{\Lambda_2}{m_2}}^d } \frac{1}{\sqrt{2 \pi \frac{\Theta_2}{m_2}}^{l_2}} \exp({- \frac{|v-u_2|^2}{2 \frac{\Lambda_2}{m_2}}}- \frac{|\eta_{l_2}|^2}{2 \frac{\Theta_2}{m_2}})
%\\
%M_{kj}(x,v,\eta_{l_k},t) &= \frac{n_{kj}}{\sqrt{2 \pi \frac{\Lambda_{kj}}{m_k}}^d } \frac{1}{\sqrt{2 \pi \frac{\Theta_{kj}}{m_k}}^{l_k}} \exp({- \frac{|v-u_{kj}|^2}{2 \frac{\Lambda_{kj}}{m_k}}}- \frac{|\eta_{l_k}- \bar{\eta}_{1_k,kj}|^2}{2 \frac{\Theta_{kj}}{m_k}}), 
%\\
%M_{21}(x,v,\eta_{l_2},t) &= \frac{n_{21}}{\sqrt{2 \pi \frac{\Lambda_{21}}{m_2}}^d } \frac{1}{\sqrt{2 \pi \frac{\Theta_{21}}{m_2}}^{l_2}} \exp({- \frac{|v-u_{21}|^2}{2 \frac{\Lambda_{21}}{m_2}}}- \frac{|\eta_{l_2}|^2}{2 \frac{\Theta_{21}}{m_2}})
\end{split}
\label{BGKmixone}
\end{align}
%for $ j,k =1,2, j \neq k$, 
where $\nu n$ is the collision frequency. % of the particles of each species with itself, while $\nu_{12} n_2$ and $\nu_{21} n_1$ are related to interspecies collisions. 
%To be flexible in choosing the relationship between the collision frequencies, we now assume the relationship
%\begin{equation} 
%\nu_{12}=\varepsilon \nu_{21}, \quad 0 < \frac{l_1}{l_1+l_2}\varepsilon \leq 1.
%\label{coll}
%\end{equation}
%The restriction $\frac{l_1}{l_1+l_2} \varepsilon \leq 1$ is without loss of generality.
%If $\frac{l_1}{l_1+l_2}\varepsilon >1$, exchange the notation $1$ and $2$ and choose $\frac{1}{\varepsilon}.$ In addition, we assume that all collision frequencies are positive. 
%For the existence and uniqueness proof we assume the following restrictions on our collision frequencies
%\begin{align}
%\nu_{jk}(x,t) n_k(x,t) = \widetilde{\nu}_{jk} \frac{n_k(x,t)}{n_1(x,t) + n_2(x,t)}, ~ j,k =1,2
%\label{asscoll}
%\end{align}
%with constants $\widetilde{\nu}_{11}, \widetilde{\nu}_{12}, \widetilde{\nu}_{21}, \widetilde{\nu}_{22}$.
We couple this kinetic equation with an algebraic equation for conservation of internal energy 
\begin{align}
\frac{d}{2} n \Lambda = \frac{d}{2} n T^{t} +\frac{l}{2} n T^{r} - \frac{l}{2} n \Theta_k,  \label{internalone}
\end{align}  
and a relaxation equation ensuring that the two temperatures $\Lambda$ and $\Theta$ relax to the same value in equilibrium
%\begin{align}
%\partial_t M_k + v \cdot \nabla_x M_k = \frac{\nu_{kk} n_k}{Z_r^k} \frac{d+l_k}{d} (\tilde{M}_k - M_k), \quad k=1,2
%\label{kin_Temp}
%\end{align}
%\textcolor{cyan}{New version: \begin{align}
% \partial_t M_k + v \cdot \nabla_x M_k = \frac{\nu_{kk} n_k}{Z_r^k} \frac{d+l_k}{d} (\tilde{M}_k - M_k)+ \frac{\nu_{kj} n_j}{Z_r^k} \frac{d+l_k}{d} (\tilde{M}_{kj} - M_k)+ \nu_{kj} n_j (M_{kj} - M_k) , \quad k=1,2 \end{align}
% }
% \textcolor{red}{
 \begin{align}
 \begin{split}
 \partial_t M + v \cdot \nabla_x M = \frac{\nu n}{Z_r} \frac{d+l}{d} (\widetilde{M} - M)&+ \nu n (M -f), \\ %\\&+ \nu_{kj} n_j (M_{kj} - f_k) 
 \Theta_k(0)= \Theta_k^0
 \end{split} 
 \label{kin_Tempone}
 \end{align}
 %}
 where $Z_r$ is a given parameter corresponding to the different rates of decays of translational and rotational/vibrational degrees of freedom. %\textcolor{cyan}{In the second term we can also choose a different $Z_r^k$}.
 Here $M$ is given by
\begin{align} 
M(x,v,\eta,t) = \frac{n}{\sqrt{2 \pi \frac{\Lambda}{m}}^d } \frac{1}{\sqrt{2 \pi \frac{\Theta}{m}}^{l}} \exp({- \frac{|v-u|^2}{2 \frac{\Lambda}{m}}}- \frac{|\eta- \bar{\eta}|^2}{2 \frac{\Theta}{m}}), %\quad k=1,2,
\label{Maxwellianone}
\end{align}
and $\widetilde{M}$ is given by 
\begin{align}
\widetilde{M}= \frac{n}{\sqrt{2 \pi \frac{T}{m}}^{d+l}} \exp \left(- \frac{m |v-u|^2}{2 T}- \frac{m|\eta_{l}- \bar{\eta}|^2}{2 T} \right), 
\label{Max_equone}
\end{align}
where $T$ is the total equilibrium temperature and is given by 
\begin{align}
T:= \frac{d \Lambda+ l \Theta}{d+l}= \frac{d T^{t} + l T^{r}}{d+l}.
\label{equ_tempone}
\end{align}
The second equality follows from \eqref{internalone}. The equation \eqref{kin_Tempone} is used to involve the temperature $\Theta$. If we multiply \eqref{kin_Tempone} by $|\eta|^2$, integrate with respect to $v$ and $\eta$ and use \eqref{equ_tempone}, we obtain  
\begin{align}
\begin{split}
\partial_t(n \Theta) +   \nabla_x\cdot (n \Theta u) = \frac{\nu n}{Z_r} n (\Lambda - \Theta)&+ \nu n n (\Theta - T^{r}) %\\ &+ \nu_{kj} n_j n_k(\Theta_{kj} - T_k^{r}).
\end{split}
\label{relaxone}
\end{align} %for $k=1,2$. 
The initial data of $\Lambda$ is determined using \eqref{internalone}. We see that in this model the term $\frac{\nu n}{Z_r} \frac{d+l}{d} (\widetilde{M} - M)$ describes the relaxation of the two temperatures $\Lambda$ and $\Theta$ to a common value. So the effect of the relaxation to common temperatures here is done by coupling the BGK equation with an additional kinetic equation, whereas in the model presented in the previous section, the property was described by a certain convex combination of the temperatures in the Maxwell distribution. \\

If we look at this model, the main differences of the model in section \ref{sec2.1} and the model here  are the following. The model in section \ref{sec2.1} has one scalar variable $I\in \mathbb{R}^+$ for all degrees of freedom in internal energy and the model here has one vector $\eta \in \mathbb{R}^l$ with one component to each degree of freedom in internal energy. If we take a component $k$, $k \in \lbrace 1,...l\rbrace$ of $\eta$, then this component squared represents the internal energy related to the internal degree of freedom $k$.  Moreover, the relaxation of the translational and rotational temperatures to a common value is done in section \ref{sec2.1} by introducing a relaxation temperature $T_{rel}$ and in the model here it is done by the additional relaxation equation \eqref{kin_Tempone}. \\

This model satisfies the following asymptotic behaviour in the space-homogeneous case. Define $\frac{1}{z} = \frac{1}{Z_r} \frac{d+l}{d}$. Then, we have
\begin{theorem}
\label{conv1one}
 Assume that $T^{r}, T^{t}, \Lambda, \Theta$ are the temperatures generated by solutions of \eqref{BGKone} coupled with \eqref{kin_Tempone} and \eqref{internalone}. Assume  that $l \geq 1$ and that
\begin{align*}
T^{r} \geq c \Theta\quad \text{with} \quad c= \frac{1}{l}(d+l)^2 \max \left\lbrace 1, \frac{A(t)}{B(t)} \right\rbrace. 
\end{align*}
Here, $A(T), B(T)$ are some constants  determined later. Moreover, assume $ \frac{1}{z} \geq  \frac{A(T)}{B(T)}$. Then, in the space homogeneous case, we have the following convergence rate of the distribution function $f$:
\begin{align*}
||f - M||_{L^1(dv)}  \leq 4 e^{-\frac{1}{2} Ct} \left(  H(f^0|\widetilde{M}^0)+ 3 z H(M^0|\widetilde{M}^0) \right)^{\frac{1}{2}}.
\end{align*}
where $C$ is given by $$C= \text{min} \left\lbrace  \nu n  ,  \frac{2c}{3 z} \nu n,  \right\rbrace.$$
\end{theorem}
We do not show the proof here, because later in section \ref{sec2.2}, we will prove it for a  model for gas mixtures which contains this model here as a special case. \\

We see that the qualitative behaviour of the speed of convergence is different for this model than for the model presented in subsection \ref{sec2.1}, since now, for one species, the constant in the exponential function depends on the minimum of the two different speed of relaxations $\nu_{11} n_1$ and $\frac{ \nu_{11} n_1}{ z_1}$, so this means, if the speed of relaxation of the two temperatures $\frac{\nu_{11} n_1}{z_1}$  is slower than the speed of relaxation to a Maxwell distribution $\nu_{11} n_1$, the rate of convergence to equilibrium is related to the parameter $z_1$, whereas, if the speed of relaxation of the two temperatures $\frac{\nu_{11} n_1}{z_1}$ times  is faster, the parameter $z_k$ has no influence and the speed of convergence is determined by the speed of relaxation to Maxwell distribution $\nu_{11} n_1$. So we see that this model seems to choose a different rate of convergence dependent on the fact which speed of relaxation is slower. 

However, for the entropy dissipation estimates, we have to assume that $z_k$ is small enough. So this model covers in principal only the regime where the speed of relaxation of the temperatures is faster than the speed of relaxation to Maxwell distributions. So it covers the opposite regime than the model described in subsection \ref{sec2.1}. 
\subsection{The BGK model of Bernard, Iollo and Puppo for one species of polyatomic molecules}
\label{2.3}
In this section we will present the model from the literature for one species described in \cite{Bernard}. For the convenience of the reader, we briefly repeat it here. For more details and motivation for the choice of the model see \cite{Bernard}. %In \cite{Pirner5} the model is presented for two species. Since gas mixtures are important in applications, we will present it here also in the case of two species. Then, in addition to a comparison with the model from the previous subsection, we can also analyse for the two species model in which way it relaxes to equilibrium.

 Let $x\in \mathbb{R}^d$ and $v\in \mathbb{R}^d, d \in \mathbb{N}$ be the phase space variables  and $t\geq 0$ the time. Let $l$ be the number of internal degrees of freedom. Further, $\eta \in \mathbb{R}^{l}$  is the variable for the internal energy degrees of freedom.
  %Let $M$ be the total number of different rotational and vibrational degrees of freedom and $l_k$ the number of internal degrees of freedom of species $k$, $k=1,2$. Note that the sum $l_1+l_2$ is not necessarily equal to $M$, because $M$ counts only the different degrees of freedom in the internal energy, $l_1+l_2$ counts all degrees of freedom in the internal energy. For example, consider two species consisting of diatomic molecules which have two rotational  degrees of freedom. In addition, the second species has one vibrational degree of freedom. Then we have $M=3, l_1=2, l_2=3$. Further, $\eta \in \mathbb{R}^{M}$  is the variable for the internal energy degrees of freedom, $\eta_{l_k} \in \mathbb{R}^{M}$ coincides with $\eta$ in the components corresponding to the internal degrees of freedom of species $k$ and is zero in the other components.  \\ Since we want to describe two different species, our kinetic model has two distribution functions $f_1(x,v,\eta_{l_1},t)> 0$ and $f_2(x,v,\eta_{l_2},t) > 0$. 
  Then the kinetic model has one distribution function $f(x,v, \eta,t) <0$.
 Furthermore, for any $f: \Lambda_{poly} \times \mathbb{R}^d \times \mathbb{R}^l \times \mathbb{R}^+_0, \Lambda_{poly} \subset \mathbb{R}^d$ with $(1+|v|^2 + |\eta|^2)f \in L^1(dv d\eta)$, $f \geq 0,$  we relate the distribution functions to  macroscopic quantities by mean-values of $f$ as it is done in \eqref{moments}.

Again, we always keep the term $\bar{\eta}$ in order to cover the most general case, but also in  \cite{Bernard}, they require $\bar{\eta}=0$.

We consider the model presented in \cite{Bernard}  given by \eqref{BGKone} with the Maxwell distribution \eqref{BGKmixone}.

%for $ j,k =1,2, j \neq k$, 
%where $\nu n$ is the collision frequencies. % of the particles of each species with itself, while $\nu_{12} n_2$ and $\nu_{21} n_1$ are related to interspecies collisions. 
%To be flexible in choosing the relationship between the collision frequencies, we now assume the relationship
%\begin{equation} 
%\nu_{12}=\varepsilon \nu_{21}, \quad 0 < \frac{l_1}{l_1+l_2}\varepsilon \leq 1.
%\label{coll}
%\end{equation}
%The restriction $\frac{l_1}{l_1+l_2} \varepsilon \leq 1$ is without loss of generality.
%If $\frac{l_1}{l_1+l_2}\varepsilon >1$, exchange the notation $1$ and $2$ and choose $\frac{1}{\varepsilon}.$ In addition, we assume that all collision frequencies are positive. 
%For the existence and uniqueness proof we assume the following restrictions on our collision frequencies
%\begin{align}
%\nu_{jk}(x,t) n_k(x,t) = \widetilde{\nu}_{jk} \frac{n_k(x,t)}{n_1(x,t) + n_2(x,t)}, ~ j,k =1,2
%\label{asscoll}
%\end{align}
%with constants $\widetilde{\nu}_{11}, \widetilde{\nu}_{12}, \widetilde{\nu}_{21}, \widetilde{\nu}_{22}$.
Again, we couple this kinetic equation with an algebraic equation for conservation of internal energy \eqref{internalone} and a relaxation equation ensuring that the two temperatures $\Lambda$ and $\Theta$ relax to the same value in equilibrium
%\begin{align}
%\partial_t M_k + v \cdot \nabla_x M_k = \frac{\nu_{kk} n_k}{Z_r^k} \frac{d+l_k}{d} (\tilde{M}_k - M_k), \quad k=1,2
%\label{kin_Temp}
%\end{align}
%\textcolor{cyan}{New version: \begin{align}
% \partial_t M_k + v \cdot \nabla_x M_k = \frac{\nu_{kk} n_k}{Z_r^k} \frac{d+l_k}{d} (\tilde{M}_k - M_k)+ \frac{\nu_{kj} n_j}{Z_r^k} \frac{d+l_k}{d} (\tilde{M}_{kj} - M_k)+ \nu_{kj} n_j (M_{kj} - M_k) , \quad k=1,2 \end{align}
% }
% \textcolor{red}{
 \begin{align}
 \begin{split}
 \partial_t M + v \cdot \nabla_x M &= \frac{\nu n}{Z_r} \frac{d+l}{d} (\widetilde{M} - M), \\%&+ \nu n (M -f), \\ %\\&+ \nu_{kj} n_j (M_{kj} - f_k) 
 \Theta_k(0)&= \Theta_k^0
 \end{split} 
 \label{kin_Temp2one}
 \end{align}
 %}
% where $Z_r$ is a given parameters corresponding to the different rates of decays of translational and rotational/vibrational degrees of freedom. %\textcolor{cyan}{In the second term we can also choose a different $Z_r^k$}.
 Here $M$ is given by \eqref{Maxwellianone} 
%\begin{align} 
%M(x,v,\eta,t) = \frac{n}{\sqrt{2 \pi \frac{\Lambda}{m}}^d } \frac{1}{\sqrt{2 \pi \frac{\Theta}{m}}^{l}} \exp({- \frac{|v-u|^2}{2 \frac{\Lambda}{m}}}- \frac{|\eta- \bar{\eta}|^2}{2 \frac{\Theta}{m}}), %\quad k=1,2,
%\label{Maxwellian}
%\end{align}
and $\widetilde{M}$ is given by \eqref{Max_equone}
%\begin{align}
%\widetilde{M}= \frac{n}{\sqrt{2 \pi \frac{T}{m}}^{d+l}} \exp \left(- \frac{m |v-u|^2}{2 T}- \frac{m|\eta_{l}- \bar{\eta}|^2}{2 T} \right), \quad k=1,2,
%\label{Max_equ}
%\end{align}
where $T$ is the total equilibrium temperature and is given by \eqref{equ_tempone}.
%\begin{align}
%T:= \frac{d \Lambda+ l \Theta}{d+l}= \frac{d T^{t} + l T^{r}}{d+l}.
%\label{equ_temp}
%\end{align}
%The second equality follows from \eqref{internal}. The equation \eqref{kin_Temp} is used to involve the temperature $\Theta$. If we multiply \eqref{kin_Temp} by $|\eta|^2$, integrate with respect to $v$ and $\eta$ and use \eqref{equ_temp}, we obtain  
%\begin{align}
%\begin{split}
%\partial_t(n \Theta) +   \nabla_x\cdot (n \Theta u) = \frac{\nu n}{Z_r} n (\Lambda - \Theta)&+ \nu n n (\Theta - T^{r}) %\\ &+ \nu_{kj} n_j n_k(\Theta_{kj} - T_k^{r}).
%\end{split}
%\label{relax}
%\end{align} %for $k=1,2$. 
%The initial data of $\Lambda$ is determined using \eqref{internal}. We see that in this model the term $\frac{\nu n}{Z_r} \frac{d+l}{d} (\widetilde{M} - M)$ describes the relaxation of the two temperatures $\Lambda$ and $\Theta$ to a common value. So the effect of the relaxation to common temperatures here is done by coupling the BGK equation with an additional kinetic equation, whereas in the model presented in the previous section, the property was described by a certain convex combination of the temperatures in the Maxwell distribution. 
So all in all, the model described in subsection \ref{2.2} and \ref{2.3} just differ in the evolution of $\Theta$ determined by the equation \eqref{kin_Tempone} and \eqref{kin_Temp2one}, respectively. In \eqref{kin_Temp2one}, the term $\nu n (M-f)$ is missing. This term was introduced by extending the model in section \ref{2.2} to gas mixtures. This can be seen later in section \ref{sec2.2}. \\
%In addition,  \eqref{BGK} and \eqref{kin_Temp} are consistent. If we multiply the equations for species $k$ of \eqref{BGK} and \eqref{kin_Temp}  by $v$ and integrate with respect to $v$ and $\eta_{l_k}$, we get in both cases for the right-hand side 
%$$ \nu_{kj} n_j n_k  (u_{jk} - u_k),$$ and if we compute the total internal energy of both equations, we obtain in both cases $$ \frac{1}{2} \nu_{kj} n_k n_j [d \Lambda_{jk} + l_j \Theta_{jk} - ( d \Lambda_j + l_j \Theta_j)].$$
%This is shown  in section 3.1 in theorem 3.2 in \cite{Pirner5}.
%\begin{remark}
%The model presented here in the case of one species does not coincide with the model in the previous subsection. The  relaxation equation of $M$ is different. Instead of \eqref{kin_Temp} in the one species case, they consider \eqref{kin_Temp2one}
%% \begin{align}
%% \begin{split}
%% \partial_t M + v \cdot \nabla_x M &= \frac{\nu n}{Z_r} \frac{d+l}{d} (\widetilde{M} - M)
%% \end{split} 
%% \label{kin_Temp2}
%% \end{align}
% so   This term was introduced by extending it to the case of gas mixtures. This can be seen in 
%\end{remark}

This model satisfies the following asymptotic behaviour in the space-homogeneous case.
\begin{theorem}
\label{conv3}
 Assume that $(f, M)$ is a solution of \eqref{BGKone} coupled with \eqref{kin_Temp2one} and \eqref{internalone}. Then, in the space homogeneous case, we have the following convergence rate of the distribution functions $f$:
 {\small
\begin{align*}
||f - \widetilde{M}||_{L^1(dv d\eta)}  \leq 4 e^{-\frac{1}{4} \widetilde{C} t} \left( H(f^0|\widetilde{M}^0)+ 2 \max \lbrace 1, z\rbrace H(M^0|\widetilde{M}^0) \right)^{\frac{1}{2}}.
\end{align*}}
where $\widetilde{C}$ is given by $$\widetilde{C}= \min \left\lbrace  \nu n  ,  \frac{\nu n}{z}  \right\rbrace.$$
\end{theorem}
We do not show the proof here, because later in section \ref{sec2}, we will prove it for a  model for gas mixtures which contains this model here as a special case. 
We observe that the qualitative behaviour of the convergence to equilibrium is the same as in the model presented in section \ref{2.2}, but we do not have to assume that $z_k$ has to be small. So this model for one species covers both regimes, the regime where the speed of relaxation to a Maxwell distribution is fast that the speed of relaxation of the two temperatures to an equal value, and the other way round. And in addition, both regimes are reflected in the convergence rate to equilibrium.
\section{Entropy dissipation estimates for a  BGK model for gas mixtures of polyatomic molecules from the literature}
\label{sec2.2}
In this section we will present a model from the literature for gas mixtures of polyatomic molecules. It is an extension of the model described in section \ref{2.2} and is presented in \cite{Pirner5}. For the convenience of the reader, we briefly repeat it here. For more details and motivation for the choice of the model see \cite{Pirner5}. For simplicity, it is presented in the case of two species. % In \cite{Pirner5} the model is presented for two species. Since gas mixtures are important in applications, we will present it here also in the case of two species. Then,  we will also analyse for this two species model in which way it relaxes to equilibrium.

 Let $x\in \mathbb{R}^d$ and $v\in \mathbb{R}^d, d \in \mathbb{N}$ be the phase space variables  and $t\geq 0$ the time. Let $M$ be the total number of different rotational and vibrational degrees of freedom and $l_k$ the number of internal degrees of freedom of species $k$, $k=1,2$. Note that the sum $l_1+l_2$ is not necessarily equal to $M$, because $M$ counts only the different degrees of freedom in the internal energy, $l_1+l_2$ counts all degrees of freedom in the internal energy. For example, consider two species consisting of diatomic molecules which have two rotational  degrees of freedom. In addition, the second species has one vibrational degree of freedom. Then we have $M=3, l_1=2, l_2=3$. Further, $\eta \in \mathbb{R}^{M}$  is the variable for the internal energy degrees of freedom, $\eta_{l_k} \in \mathbb{R}^{M}$ coincides with $\eta$ in the components corresponding to the internal degrees of freedom of species $k$ and is zero in the other components.  \\ Since we want to describe two different species, our kinetic model has two distribution functions $f_1(x,v,\eta_{l_1},t)> 0$ and $f_2(x,v,\eta_{l_2},t) > 0$. 
 Furthermore, for any $f_1,f_2: \Lambda_{poly} \times \mathbb{R}^d \times \mathbb{R}^M \times \mathbb{R}^+_0, \Lambda_{poly} \subset \mathbb{R}^d$ with $(1+|v|^2 + |\eta_{l_k}|^2)f_k \in L^1(dv d\eta_{l_k})$, $f_1,f_2 \geq 0,$  we relate the distribution functions to  macroscopic quantities by mean-values of $f_k$, $k=1,2$ as follows
\begin{align}
\int f_k(v, \eta_{l_k}) \begin{pmatrix}
1 \\ v \\ \eta_{l_k} \\ m_k |v-u_k|^2 \\ m_k |\eta_{l_k} - \bar{\eta}_{l_k} |^2 \\ m_k (v-u_k(x,t)) \otimes (v-u_k(x,t))
\end{pmatrix} 
dv d\eta_{l_k}=: \begin{pmatrix}
n_k \\ n_k u_k \\ n_k \bar{\eta}_{l_k} \\ d n_k T_k^{t} \\ l_k n_k T_k^{r} \\ \mathbb{P}_k
\end{pmatrix} , 
\label{moments}
\end{align} 
for $k=1,2$, where $n_k$ is the number density, $u_k$ the mean velocity, $\bar{\eta}_{l_k}$ the mean variable related to the internal energy, $T_k^{t}$ the mean temperature of the translation, $T_k^{r}$ the mean temperature of the internal energy degrees of freedom for example rotation or vibration and $\mathbb{P}_k$ the pressure tensor of species $k$, $k=1,2$. Note that in this paper we shall write $T_k^{t}$ and $T_k^{r}$ instead of $k_B T_k^{t}$ and $k_B T_k^{r}$, where $k_B$ is Boltzmann's constant. 
In the following, we always keep the term $\bar{\eta}_{l_k}$ in order to cover the most general case. % but in \cite{Bernard} and \cite{Pirner5}, they require $\bar{\eta}_{l_k}=0$, which means that the energy in rotations clockwise is the same as in rotations counter clockwise. Similar for vibrations, whereas in \cite{Pirner8}, it is shown that if one requires $\bar{\eta}_{l_k} = \omega$ with a fixed $\omega \in \mathbb{R}^M$ such that $|\omega|^2 = 2 \frac{p_{\infty}}{m n}$ leads to a more general equation of state in equilibrium given by $p= n T + const.$

We consider the model presented in \cite{Pirner5}  given by

\begin{align} \begin{split} \label{BGK}
\partial_t f_1 + v\cdot\nabla_x   f_1   &= \nu_{11} n_1 (M_1 - f_1) + \nu_{12} n_2 (M_{12}- f_1),
\\ 
\partial_t f_2 + v\cdot\nabla_x   f_2 &=\nu_{22} n_2 (M_2 - f_2) + \nu_{21} n_1 (M_{21}- f_2), \\
f_1(t=0) &= f_1^0, \\
f_2(t=0) &= f_2^0
\end{split}
\end{align}
with the Maxwell distributions
\begin{align} 
\begin{split}
M_k(x,v,\eta_{l_k},t) &= \frac{n_k}{\sqrt{2 \pi \frac{\Lambda_k}{m_k}}^d } \frac{1}{\sqrt{2 \pi \frac{\Theta_k}{m_k}}^{l_k}} \exp({- \frac{|v-u_k|^2}{2 \frac{\Lambda_k}{m_k}}}- \frac{|\eta_{l_k}- \bar{\eta}_{l_k}|^2}{2 \frac{\Theta_k}{m_k}}), 
\\
%M_2(x,v,\eta_{l_2},t) &= \frac{n_2}{\sqrt{2 \pi \frac{\Lambda_2}{m_2}}^d } \frac{1}{\sqrt{2 \pi \frac{\Theta_2}{m_2}}^{l_2}} \exp({- \frac{|v-u_2|^2}{2 \frac{\Lambda_2}{m_2}}}- \frac{|\eta_{l_2}|^2}{2 \frac{\Theta_2}{m_2}})
%\\
M_{kj}(x,v,\eta_{l_k},t) &= \frac{n_{kj}}{\sqrt{2 \pi \frac{\Lambda_{kj}}{m_k}}^d } \frac{1}{\sqrt{2 \pi \frac{\Theta_{kj}}{m_k}}^{l_k}} \exp({- \frac{|v-u_{kj}|^2}{2 \frac{\Lambda_{kj}}{m_k}}}- \frac{|\eta_{l_k}- \bar{\eta}_{1_k,kj}|^2}{2 \frac{\Theta_{kj}}{m_k}}), 
%\\
%M_{21}(x,v,\eta_{l_2},t) &= \frac{n_{21}}{\sqrt{2 \pi \frac{\Lambda_{21}}{m_2}}^d } \frac{1}{\sqrt{2 \pi \frac{\Theta_{21}}{m_2}}^{l_2}} \exp({- \frac{|v-u_{21}|^2}{2 \frac{\Lambda_{21}}{m_2}}}- \frac{|\eta_{l_2}|^2}{2 \frac{\Theta_{21}}{m_2}})
\end{split}
\label{BGKmix}
\end{align}
for $ j,k =1,2, j \neq k$, 
where $\nu_{11} n_1$ and $\nu_{22} n_2$ are the collision frequencies of the particles of each species with itself, while $\nu_{12} n_2$ and $\nu_{21} n_1$ are related to interspecies collisions. 
To be flexible in choosing the relationship between the collision frequencies, we now assume the relationship
\begin{equation} 
\nu_{12}=\varepsilon \nu_{21}, \quad 0 < \frac{l_1}{l_1+l_2}\varepsilon \leq 1.
\label{coll}
\end{equation}
The restriction $\frac{l_1}{l_1+l_2} \varepsilon \leq 1$ is without loss of generality.
If $\frac{l_1}{l_1+l_2}\varepsilon >1$, exchange the notation $1$ and $2$ and choose $\frac{1}{\varepsilon}.$ In addition, we assume that all collision frequencies are positive. 
For the existence and uniqueness proof we assume the following restrictions on our collision frequencies
\begin{align}
\nu_{jk}(x,t) n_k(x,t) = \widetilde{\nu}_{jk} \frac{n_k(x,t)}{n_1(x,t) + n_2(x,t)}, ~ j,k =1,2
\label{asscoll}
\end{align}
with constants $\widetilde{\nu}_{11}, \widetilde{\nu}_{12}, \widetilde{\nu}_{21}, \widetilde{\nu}_{22}$.
We couple these kinetic equations with an algebraic equation for conservation of internal energy 
\begin{align}
\frac{d}{2} n_k \Lambda_k = \frac{d}{2} n_k T_k^{t} +\frac{l_k}{2} n_k T_k^{r} - \frac{l_k}{2} n_k \Theta_k, \quad k=1,2. \label{internal}
\end{align}  
and a relaxation equation ensuring that the two temperatures $\Lambda_k$ and $\Theta_k$ relax to the same value in equilibrium
%\begin{align}
%\partial_t M_k + v \cdot \nabla_x M_k = \frac{\nu_{kk} n_k}{Z_r^k} \frac{d+l_k}{d} (\tilde{M}_k - M_k), \quad k=1,2
%\label{kin_Temp}
%\end{align}
%\textcolor{cyan}{New version: \begin{align}
% \partial_t M_k + v \cdot \nabla_x M_k = \frac{\nu_{kk} n_k}{Z_r^k} \frac{d+l_k}{d} (\tilde{M}_k - M_k)+ \frac{\nu_{kj} n_j}{Z_r^k} \frac{d+l_k}{d} (\tilde{M}_{kj} - M_k)+ \nu_{kj} n_j (M_{kj} - M_k) , \quad k=1,2 \end{align}
% }
% \textcolor{red}{
 \begin{align}
 \begin{split}
 \partial_t M_k + v \cdot \nabla_x M_k = \frac{\nu_{kk} n_k}{Z_r^k} \frac{d+l_k}{d} (\widetilde{M}_k - M_k)&+ \nu_{kk} n_k (M_k -f_k) \\&+ \nu_{kj} n_j (M_{kj} - f_k) , \\
 \Theta_k(0)= \Theta_k^0
 \end{split} 
 \label{kin_Temp}
 \end{align}
 %}
for $j,k=1,2, j \neq k$, where $Z_r^k$ are given parameters corresponding to the different rates of decays of translational and rotational/vibrational degrees of freedom. %\textcolor{cyan}{In the second term we can also choose a different $Z_r^k$}.
 Here $M_k$ is given by
\begin{align} 
M_k(x,v,\eta_{l_k},t) = \frac{n_k}{\sqrt{2 \pi \frac{\Lambda_k}{m_k}}^d } \frac{1}{\sqrt{2 \pi \frac{\Theta_k}{m_k}}^{l_k}} \exp({- \frac{|v-u_k|^2}{2 \frac{\Lambda_k}{m_k}}}- \frac{|\eta_{l_k}- \bar{\eta}_{l_k}|^2}{2 \frac{\Theta_k}{m_k}}), \quad k=1,2,
\label{Maxwellian}
\end{align}
and $\widetilde{M}_k$ is given by 
\begin{align}
\widetilde{M}_k= \frac{n_k}{\sqrt{2 \pi \frac{T_k}{m_k}}^{d+l_k}} \exp \left(- \frac{m_k |v-u_k|^2}{2 T_k}- \frac{m_k|\eta_{l_k}- \bar{\eta}_{l_k,kj}|^2}{2 T_k} \right), \quad k=1,2,
\label{Max_equ}
\end{align}
where $T_k$ is the total equilibrium temperature and is given by 
\begin{align}
T_k:= \frac{d \Lambda_k + l_k \Theta_k}{d+l_k}= \frac{d T^{t}_k + l_k T^{r}_k}{d+l_k}.
\label{equ_temp}
\end{align}
The second equality follows from \eqref{internal}. The equation \eqref{kin_Temp} is used to involve the temperature $\Theta_k$. If we multiply \eqref{kin_Temp} by $|\eta_{l_k}|^2$, integrate with respect to $v$ and $\eta_{l_k}$ and use \eqref{equ_temp}, we obtain  
\begin{align}
\begin{split}
\partial_t(n_k \Theta_k) +   \nabla_x\cdot (n_k \Theta_k u_k) = \frac{\nu_{kk} n_k}{Z_r^k} n_k (\Lambda_k - \Theta_k)&+ \nu_{kk} n_k n_k (\Theta_k - T_k^{r}) \\ &+ \nu_{kj} n_j n_k(\Theta_{kj} - T_k^{r}).
\end{split}
\label{relax}
\end{align} for $k=1,2$. The initial data of $\Lambda_k$ is determined using \eqref{internal}. We see that in this model the term $\frac{\nu_{kk} n_k}{Z_r^k} \frac{d+l_k}{d} (\widetilde{M}_k - M_k)$ describes the relaxation of the two temperatures $\Lambda_k$ and $\Theta_k$ to a common value. So again,  the effect of the relaxation to common temperatures here is done by coupling the BGK equation with an additional kinetic equation. % whereas in the model presented in the previous section, the property was described by a certain convex combination of the temperatures in the Maxwell distribution. 

In addition,  \eqref{BGK} and \eqref{kin_Temp} are consistent. If we multiply the equations for species $k$ of \eqref{BGK} and \eqref{kin_Temp}  by $v$ and integrate with respect to $v$ and $\eta_{l_k}$, we get in both cases for the right-hand side 
$$ \nu_{kj} n_j n_k  (u_{jk} - u_k),$$ and if we compute the total internal energy of both equations, we obtain in both cases $$ \frac{1}{2} \nu_{kj} n_k n_j [d \Lambda_{jk} + l_j \Theta_{jk} - ( d \Lambda_j + l_j \Theta_j)].$$
This is shown  in section 3.1 in theorem 3.2 in \cite{Pirner5}. This was the reason for adding the additional term $\nu_{kj} n_j (M_{kj} - f_k)$ in \eqref{kin_Temp}. Then, in order to treat the two types of interactions, interaction of a species with itself and interactions with the other species, the term $\nu_{kk} n_k (M_k-f_k)$ was added, too, even in the one species model. This is the reason, why the model in section \ref{2.2} differs from the model in \ref{2.3}.

In the following we will recall some properties of this model which we will need later for the proof of the convergence rate to equilibrium.

\subsection{Conservation properties}
The Maxwell distributions $M_1$ and $M_2$ in \eqref{BGKmix} have the same densities, mean velocities and internal energies as $f_1$ and $f_2$, respectively. With this choice, we guarantee the conservation of the number of particles, momentum and internal energy in interactions of one species with itself (see section 3.2 in \cite{Pirner5}). The remaining parameters $n_{12}, n_{21}, u_{12}, u_{21}, \Lambda_{12}, \Lambda_{21}, \Theta_{12}$ and $\Theta_{21}$ will be determined using conservation of the number of particles, of total momentum and total energy, together with some symmetry considerations. 

If we assume that \begin{align} n_{12}=n_1 \quad \text{and} \quad n_{21}=n_2,  
\label{density2} 
\end{align}
we have conservation of the number of particles, see theorem 2.1 in \cite{Pirner}.
If we further assume 
 \begin{align}
u_{12}&= \delta u_1 + (1- \delta) u_2, \quad \delta \in \mathbb{R},
\label{convexvel2}
\end{align} then we have conservation of total momentum
provided that
\begin{align}
u_{21}=u_2 - \frac{m_1}{m_2} \varepsilon (1- \delta ) (u_2 - u_1), 
\label{veloc2}
\end{align}
see theorem 2.2 in \cite{Pirner}. 
%In \cite{Pirner5} where this model is introduced it is assumed that $\bar{\eta}_{l_1} = \bar{\eta}_{l_2}=0$. In order to give a proof for the most general case, we do not make this assumption. If we do not make this assumption,
We also need corresponding definitions for $\bar{\eta}_{l_1,12}$ and $\bar{\eta}_{l_2,21}$. This is done in the next definition.
\begin{defi}
\label{defeta}
We consider
\begin{align*}
\bar{\eta}_{12} &= \beta \bar{\eta}_{l_1} + (1- \beta) \bar{\eta}_{l_2}, \quad \beta \in \mathbb{R},
\end{align*}
and define $\bar{\eta}_{l_1,12}$ as the vector which is equal to $\bar{\eta}_{12}$ in the components where $\eta_{l_1}$ coincides with $\eta$ and zero otherwise. In addition, consider
\begin{align*}
\bar{\eta}_{21}=\bar{\eta}_{l_2} - \frac{m_1}{m_2} \varepsilon (1- \beta ) (\bar{\eta}_{l_2} - \bar{\eta}_{l_1})
\end{align*}
and define $\bar{\eta}_{l_2,21}$ as the vector which is equal to $\bar{\eta}_{21}$ in the components where $\eta_{l_2}$ coincides with $\eta$ and zero otherwise.
\end{defi}
Similar as in the case of the mean velocities one can prove that this definition leads to conservation of momentum.
If we further assume that $\Lambda_{12}$ and $\Theta_{12}$ are of the following form
\begin{align}
\begin{split}
\Lambda_{12} &=  \alpha \Lambda_1 + ( 1 - \alpha) \Lambda_2 + \gamma |u_1 - u_2 | ^2,  \quad 0 \leq \alpha \leq 1, \gamma \geq 0 ,\\
\Theta_{12} &= \frac{l_1 \Theta_1 + l_2 \Theta_2}{l_1 +l_2} + \tilde{\gamma} | \bar{\eta}_{l_1} - \bar{\eta}_{l_2}|^2, \quad \quad \quad \quad \quad \quad \quad ~ \tilde{\gamma} \geq 0,
\label{contemp2}
\end{split}
\end{align}
then we have conservation of total energy and a uniform choice of the temperatures 
provided that
\begin{align}
\begin{split}
\Lambda_{21} &=\left[ \frac{1}{d} \varepsilon m_1 (1- \delta) \left( \frac{m_1}{m_2} \varepsilon ( \delta - 1) + \delta +1 \right) - \varepsilon \gamma \right] |u_1 - u_2|^2 \\&+ \varepsilon ( 1 - \alpha ) \Lambda_1 + ( 1- \varepsilon ( 1 - \alpha)) \Lambda_2, \\
\Theta_{21}&= \varepsilon \frac{l_1}{l_1 +l_2} \Theta_1 + \left( 1- \varepsilon \frac{l_1}{l_1+l_2} \right) \Theta_2 - \frac{l_1}{l_2} \varepsilon \tilde{\gamma} |\bar{\eta}_{l_1} - \bar{\eta}_{l_2}|^2 \\&- \varepsilon \frac{m_1}{l_2} (|\bar{\eta}_{l_1,12}|^2 - |\bar{\eta}_{l_1}|^2) - \frac{m_2}{l_2} (|\bar{\eta}_{l_2,21}|^2 -|\bar{\eta}_{l_2}|^2)
\label{temp2}
\end{split}
\end{align}
see theorem 3.2 and remark 3.2 in \cite{Pirner5}.
In order to ensure the positivity of all temperatures, we need to restrict $\delta$, $\beta$, $\gamma$ and $\tilde{\gamma}$ to 
 \begin{align}
 \begin{split}
0 \leq \gamma  \leq \frac{m_1}{d} (1-\delta) \left[(1 + \frac{m_1}{m_2} \varepsilon ) \delta + 1 - \frac{m_1}{m_2} \varepsilon \right], \\
0 \leq \tilde{\gamma}  \leq \frac{m_1}{l_1} (1-\beta) \left[(1 + \frac{m_1}{m_2} \varepsilon ) \beta + 1 - \frac{m_1}{m_2} \varepsilon \right], 
\end{split}
 \label{gamma2}
 \end{align}
and
\begin{align}
\begin{split}
 \frac{ \frac{m_1}{m_2}\varepsilon - 1}{1+\frac{m_1}{m_2}\varepsilon} \leq  \delta \leq 1, \quad
  \frac{ \frac{m_1}{m_2}\varepsilon - 1}{1+\frac{m_1}{m_2}\varepsilon} \leq  \beta \leq 1,
  \end{split}
\label{gammapos2}
\end{align}
see theorem 2.5 in \cite{Pirner} for $N=3$ in the mono-atomic case. 

%In this chapter, we took \cite{Bernard} as basis to extend it to mixtures. 
%\begin{remark}
%Now, we can motivate the difference of the models for one species in subsection \ref{2.2} and \ref{2.3}. The additional term $\nu n (M-f)$ in the relaxation equation was added in the case of gas mixtures
%in order to ensure that the time evolution of the mean velocity is the same using \eqref{BGK} or \eqref{kin_Temp}, so  the term $\nu_{kj} n_j (M_{kj}-f_k)$ is added. And in order to treat the two types of collisions, collisions of a species with itself and collisions with the other species, in the same way, the term $\nu_{kk} n_k (M_k-f_k)$ was added, too.
%\end{remark}

\subsection{Existence and uniqueness of mild solutions}
The existence and uniqueness of mild solutions of the model presented here, was established in \cite{Pirner8}. 
\begin{theorem}
Under certain assumption on the initial data, the domain in space and the collision frequencies (see \cite{Pirner8} for the details), the definitions \eqref{moments}, \eqref{density2}, \eqref{convexvel2}, \eqref{veloc2},  \eqref{contemp2} and \eqref{temp2}, there exists a unique non-negative mild solution \\ $(f_1,f_2, g_1, g_2)\in C(\mathbb{R}^+ ; L^1((1+ |v|^2) dv dx))$ of the initial value problem \eqref{BGK} coupled with \eqref{kin_Temp} and \eqref{internal}. Moreover, for all $t>0$ the following bounds hold:
\begin{align*}
|u_k(t)|, |u_{12}(t)|, |u_{21}(t)|, |\eta_{l_k}|, |\bar{\eta}_{l_k}|, |\bar{\eta}_{l_k}|, T_k(t), T_{12}(t), T_{21}(t), N_q(f_k)(t) \leq A(t) &< \infty, \\
n_k(t) \geq C_0 e^{-(\widetilde{\nu}_{kk} + \widetilde{\nu}_{kj})t} &>0, \\
T_k(t), \Lambda_k(t), \Theta_k(t), \Lambda_{12}(t), \Theta_{12}(t), \Lambda_{21}(t), \Theta_{21}(t) \geq B(t)&>0,
\end{align*} 
for $k=1,2$ and some constants $A(t),B(t)$ given by
$$ A(t) = C e^{Ct}, ~ B(t) = C e^{-Ct}, ~ C>0$$
\label{ex2}
\end{theorem}
\subsection{Equilibrium, H-theorem and entropy inequality}
In \cite{Pirner5}, the following characterization of equilibrium and the H-theorem is proven.
\begin{theorem}[Equilibrium]
Assume $f_1, f_2 >0$ with $f_1$ and $f_2$ independent of $x$ and $t$.
Assume the conditions \eqref{density2}, \eqref{convexvel2}, \eqref{veloc2}, \eqref{contemp2} and \eqref{temp2}, $\delta \neq 1, \alpha \neq 1, l_1,l_2\neq 0$, so that all temperatures are positive. %especially \eqref{gamma}

Then $f_1$ and $f_2$ are Maxwell distributions with equal mean velocities $u:=u_1=u_2=u_{12}=u_{21}$ and temperatures $T:=T_1^{r}=T_2^{r}=T_1^{t}=T_2^{t}=\Lambda_1=\Lambda_2=\Theta_1=\Theta_2=\Theta_{12}=\Theta_{21}= \Lambda_{12}=\Lambda_{21}$. This means $f_k$ is given by
$$ M_k(x,v,\eta_{l_k},t) = \frac{n_k}{\sqrt{2 \pi \frac{T}{m_k}}^d } \frac{1}{\sqrt{2 \pi \frac{T}{m_k}}^{l_k}} \exp({- \frac{|v-u|^2}{2 \frac{T}{m_k}}}- \frac{|\eta_{l_k}- \bar{\eta}_{l_k}|^2}{2 \frac{T}{m_k}}), \quad k=1,2.$$
\label{equilibrium}
\end{theorem}
\begin{lemma}[Contribution to the H-theorem from the one species relaxation terms]
Assume $f_1, f_2 >0$. % and define $\frac{1}{z_k}:= \frac{1}{Z_k^r} \frac{d+l_k}{d}, k=1,2$. %Assume $z_k \leq 1$. 
Then
$$ \int \ln f_k (M_k - f_k) dv d\eta_{l_k} + %\frac{1}{z_k} 
\int \ln M_k ( \widetilde{M}_k - M_k) dv d\eta_{l_k} \leq 0, \quad k=1,2, $$ with equality if and only if $M_k=f_k$ and  $\Lambda_k=\Theta_k=T_k^{r} = T_k^{t}$.
\label{one_species}
\end{lemma}
\begin{theorem}[H-theorem for mixture]
Assume $f_1, f_2 >0$. Assume %$z_1,z_2 \leq 1$ and 
$\nu_{11} n_1 \geq \nu_{12} n_2$, $\nu_{22} n_2 \geq \nu_{21} n_1$, $\alpha, \delta\neq 1, l_1,l_2 \neq 0.$
Assume the relationship between the collision frequencies \eqref{coll}, the conditions for the interspecies Maxwellians \eqref{density2}, \eqref{convexvel2}, \eqref{veloc2}, \eqref{contemp2} and \eqref{temp2} and the positivity of all temperatures, then
\begin{align*}
\sum_{k=1}^2 [ \nu_{kk} n_k \int (M_k - f_k) \ln f_k dv d\eta_{l_k} + %\frac{
\nu_{kk} n_k%}{z_k}
 \int (\widetilde{M}_k - M_k) \ln M_k dv d\eta_{l_k}] \\+ %\frac{
 \nu_{11} n_1%}{z_1}
  \int (\widetilde{M}_1 - M_1) \ln M_1 dv d\eta_{l_1} + %\frac{
  \nu_{22} n_2%}{z_2} 
  \int (\widetilde{M}_2 - M_2) \ln M_2 dv d\eta_{l_2} \\ +\nu_{12} n_2 \int (M_{12} - f_1) \ln f_1 dv d\eta_{l_1} + \nu_{21} n_1 \int (M_{21} - f_2) \ln f_2 dv d\eta_{l_2}  \leq 0,
\end{align*}
with equality if and only if $f_1$ and $f_2$ are Maxwell distributions with equal mean velocities and all temperatures coincide.
\label{H-theorem}
\end{theorem}
Define %the function $g_k= \hat{G}_k -f_k$ and 
$\frac{1}{z_k} := \frac{1}{Z_k^r} \frac{d+l_k}{d}, ~ k=1,2$ and the
 total entropy \begin{align*}
 H(f_1,f_2) = \int (f_1 \ln f_1  + 3 z_1 M_1 \ln M_1) dv d\eta_{l_1} + \int (f_2 \ln f_2+ 3 z_2 M_2 \ln M_2 ) dvd\eta_{l_2}.
% \label{H-funct}
 \end{align*}
  \begin{cor}[Entropy inequality for mixtures]
Assume $f_1, f_2 >0$,  $\Lambda_k$ and $\Theta_k$ are bounded from below and above and $T_k^{r} \geq \widetilde{C} \Theta_k$ for an appropriate $\widetilde{C}$ and $z_k$ small enough.
Assume relationship \eqref{coll}, the conditions \eqref{density2}, \eqref{convexvel2}, \eqref{veloc2}, \eqref{contemp2} and \eqref{temp2} and the positivity of all temperatures, then we have the following entropy inequality
\begin{align*}
\partial_t \left(  H(f_1,f_2) \right) \\+ \nabla_x \cdot \left(\int  v (f_1 \ln f_1  + 3 z_1 M_1 \ln M_1) dv d\eta_{l_1} + \int v( f_2 \ln f_2+ 3 z_2 M_2 \ln M_2 ) dv d\eta_{l_2} \right) \leq 0,
\end{align*}
with equality if and only if $f_1$ and $f_2$ are Maxwell distribution and all temperatures coincide.
\end{cor}
\subsection{Entropy dissipation estimates and convergence rate to equilibrium}
In the following we will prove three lemmas which will be used to prove the convergence rate to equilibrium for the model in this subsection.
\begin{lemma}
Assume that $f_1,f_2 >0$. Let $\widetilde{M}_k$ be the Maxwellian defined by \eqref{Max_equ} and $\widetilde{M}_{kj}$ the Maxwellian defined in \eqref{Max_equ3}. Then \begin{align*}
\int \widetilde{M}_k \ln \widetilde{M}_k dv d\eta_{l_k} \leq \int M_k \ln M_k dv d\eta_{l_k},  
\end{align*} 
for $k,j=1,2, ~ k\neq j$.
\label{Mtilde}
\end{lemma}
\begin{proof}
Using that $\ln M_k = \ln(\frac{n_k}{\sqrt{2 \pi \frac{\Lambda_k}{m_k}}^n} \frac{1}{\sqrt{2 \pi \frac{\Theta_k}{m_k}}^{l_k}}) - \frac{|v-u_k|^2}{2 \frac{\Lambda_k}{m_k}} - \frac{|\eta_{l_k}|^2}{2 \frac{\Theta_k}{m_k}}$ and $\ln \widetilde{M}_k = \ln(\frac{n_k}{\sqrt{2 \pi \frac{\frac{1}{n+l_k}(d \Lambda_k + l_k \Theta_k)}{m_k}}^{l_k+d}}) - \frac{|v-u_k|^2}{2 \frac{\frac{1}{d+l_k}(d \Lambda_k + l_k \Theta_k)}{m_k}} - \frac{|\eta_{l_k}|^2}{2 \frac{\frac{1}{d+l_k}(d \Lambda_k + l_k \Theta_k)}{m_k}}$, we compute the integrals and obtain that the required inequality is equivalent to
\begin{align*}
\ln(\frac{n_k}{\sqrt{2 \pi \frac{\frac{1}{d+l_k}(d \Lambda_k + l_k \Theta_k)}{m_k}}^{d+l_k}} ) &\leq \ln(\frac{n_k}{\sqrt{2 \pi \frac{\Lambda_k}{m_k}}^d} \frac{1}{\sqrt{2 \pi \frac{\Theta_k}{m_k}}^{l_k}}) %\\
%\ln(\frac{n_k}{\sqrt{2 \pi \frac{\frac{1}{d+l_k}(d \Lambda_k + l_k \Theta_k)}{m_k}}^{d+l_k}} ) &\leq \ln(\frac{n_k}{\sqrt{2 \pi \frac{T^{trans}_k}{m_k}}^d} \frac{1}{\sqrt{2 \pi \frac{T^{rot}_k}{m_k}}^{l_k}})
\end{align*}  which is equivalent to the condition
\begin{align*}
\frac{d+l_k}{2} \ln(\frac{1}{d+l_k}(d \Lambda_k + l_k \Theta_k)) &\geq \frac{d}{2} \ln \Lambda_k + \frac{l_k}{2}\ln \Theta_k %\\
%\frac{d+l_k}{2} \ln(\frac{1}{d+l_k}(d \Lambda_k + l_k \Theta_k)) &\geq \frac{d}{2} \ln T_k^{trans} + \frac{l_k}{2}\ln T_k^{rot} 
\end{align*} 
This is true since $\ln$ is a concave function. %In the second one we replace $n \Lambda_k + l_k \Theta_k$ by $n T_k^{trans} + l_k T_k^{rot}$ using \eqref{internal} and again we use that $\ln$ is concave.
\end{proof}
\begin{lemma}
Assume that $T_k^{r}, T_k^{t}, \Lambda_k, \Theta_k$ are the temperatures generated by solutions of \eqref{BGK} coupled with \eqref{kin_Temp} and \eqref{internal}. Assume that $l_k \geq 1$ and that
\begin{align}
T_k^{r} \geq c_k \Theta_k\quad \text{with} \quad c_k= \frac{1}{l_k}(d+l_k)^2 \max \left\lbrace 1, \frac{A(t)}{B(t)} \right\rbrace, 
\label{asstemp}
\end{align}
with $A$ and $B$ specified in theorem \ref{ex2}.
Then we have 
\begin{align}
\frac{d}{2} \frac{T_k}{\Lambda_k} + \frac{l_k}{2} \frac{T_k}{\Theta_k} \geq \frac{d+l_k}{2}
\end{align}
\begin{align}
\begin{split}
\nu_{kk} n_k \left(\frac{d}{2} \frac{T_k^{t}}{\Lambda_k} + \frac{l_k}{2} \frac{T_k^{r}}{\Theta_k}\right) + \nu_{kj} n_j \left( \frac{d}{2} \frac{T_k^{t}}{\Lambda_k} + \frac{l_k}{2} \frac{T_k^{r}}{\Theta_k}\right) \\\geq \nu_{kk} n_k \frac{d+l_k}{2} + \nu_{kj} n_j \left( \frac{d}{2} \frac{\Lambda_{kj}}{\Lambda_k} + \frac{l_k}{2} \frac{\Theta_{kj}}{\Theta_k} \right)
\end{split}
\end{align}
\label{lemm1}
\end{lemma}
\begin{proof}
For the first estimate we use the positivity of the term $\frac{d}{2}\frac{T_k}{\Lambda_k}$, assumption \eqref{asstemp}, the definition of $T_k$ given by \eqref{equ_temp}, the positivity of $T_k^r$ and the assumption $l_k \geq 1$ to obtain
$$ \frac{d}{2} \frac{T_k}{\Lambda_k} + \frac{l_k}{2} \frac{T_k}{\Theta_k} \geq \frac{l_k}{2} \frac{T_k}{\frac{1}{c_k} T_k^{r}} \geq \frac{l_k^2}{2} \frac{c_k}{l_k+d} = \frac{l_k}{2} (d+l_k) \max \left\lbrace 1, \frac{A(t)}{B(t)} \right\rbrace \geq \frac{d+l_k}{2}$$
For the second estimate, we first use the estimates of theorem \ref{ex2} to obtain
\begin{align}
\begin{split}
\nu_{kk} n_k \frac{d+l_k}{2} &+ \nu_{kj} n_j \left( \frac{d}{2} \frac{\Lambda_{kj}}{\Lambda_k} + \frac{l_k}{2} \frac{\Theta_{kj}}{\Theta_k} \right) \\&\leq \nu_{kk} n_k \frac{d+l_k}{2} + \nu_{kj} n_j \left( \frac{d}{2} \frac{A(t)}{B(t)} + \frac{l_k}{2} \frac{A(t)}{	B(t)} \right) \\ &\leq (\nu_{kk} n_k + \nu_{kj} n_j ) \max \left\lbrace 1, \frac{A(t)}{B(t)} \right\rbrace \frac{d+l_k}{2} 
\end{split}
\label{I}
\end{align}
and using the positivity of the term $\frac{d}{2} \frac{T_k^{t}}{\Lambda_k}$ and assumption \eqref{asstemp} leads to
\begin{align*}
(\nu_{kk} n_k + \nu_{kj} n_j) \left(\frac{d}{2} \frac{T_k^{t}}{\Lambda_k} + \frac{l_k}{2} \frac{T_k^{r}}{\Theta_k} \right) \geq (\nu_{kk} n_k + \nu_{kj} n_j) \frac{l_k}{2} c_k \\= (\nu_{kk} n_k + \nu_{kj} n_j) \frac{(d+l_k)^2}{2} \max \left\lbrace 1, \frac{A(t)}{B(t)}\right\rbrace 
\end{align*}
Since this is larger than \eqref{I}, we get the claimed inequality.
\end{proof}
 We denote by $H_k(f_k)= \int \int f_k \ln f_k dv d\eta_{l_k}$ the entropy of a function $f_k$ and by $H_k(f_k|g_k)= \int \int f_k \ln \frac{f_k}{g_k} dv d\eta_{l_k}$ the relative entropy of $f_k$ and $g_k$.
 \begin{lemma}
 Assume that $T_k^{r}, T_k^{t}, \Lambda_k, \Theta_k$ are the temperatures generated by solutions of \eqref{BGK} coupled with \eqref{kin_Temp} and \eqref{internal}. Assume the condition \eqref{asscoll} and that $l_k \geq 1$ and assume the inequality \eqref{asstemp}. Let $t \in [0,T].$ Then, if 
 $$ \frac{\nu_{kk} n_k}{z_k} \geq (\widetilde{\nu}_{kk} + \widetilde{\nu}_{kj}) \frac{A(T)}{B(T)},$$
 $z_k$ satisfies
 {\small
 \begin{align}
  z_k \leq \frac{\nu_{kk} n_k (\frac{d}{2} \frac{T_k}{\Lambda_k} + \frac{l_k}{2} \frac{T_k}{\Theta_k} - \frac{d+l_k}{2})}{\nu_{kk} n_k \left(\frac{d}{2} \frac{T_k^{t}}{\Lambda_k} + \frac{l_k}{2} \frac{T_k^{r}}{\Theta_k}\right) + \nu_{kj} n_j \left( \frac{d}{2} \frac{T_k^{t}}{\Lambda_k} + \frac{l_k}{2} \frac{T_k^{r}}{\Theta_k}\right)- \nu_{kk} n_k \frac{d+l_k}{2} - \nu_{kj} n_j \left( \frac{d}{2} \frac{\Lambda_{kj}}{\Lambda_k} + \frac{l_k}{2} \frac{\Theta_{kj}}{\Theta_k} \right)}
  \label{inz}
  \end{align}}
  \label{lemm2}
 \end{lemma}
 \begin{proof}
 According to lemma \ref{lemm1} the claimed upper bound of $z_k$ is non-negative. The proof of lemma \ref{lemm1} even leads to the better estimate
 $$ z_k \leq \frac{\widetilde{\nu}_{kk} n_k T_k}{(\widetilde{\nu}_{kk} n_k + \widetilde{\nu}_{kj} n_j) T_k^{t}} $$
 by not estimating the terms $\frac{d}{2}\frac{T_k}{\Lambda_k}$ and $\frac{d}{2} \frac{T_k^{t}}{\Lambda_k}$ by zero from below. Note, that we also used the condition \eqref{asscoll}.
 Using the estimates from theorem \ref{ex2}, we guarantee \eqref{inz} if $z_k$ satisfies
 $$ z_k \leq \frac{\widetilde{\nu}_{kk} n_k B(t)}{(\widetilde{\nu}_{kk} n_k + \widetilde{\nu}_{kj} n_j) A(t)}$$
 which is equivalent to the assumed inequality 
 $$ \frac{\nu_{kk} n_k}{z_k} \geq (\widetilde{\nu}_{kk} + \widetilde{\nu}_{kj}) \frac{A(T)}{B(T)},$$
 since $\frac{A(t)}{B(t)}$ is monotone increasing and $\frac{n_k}{n_k+n_j}, \frac{n_j}{n_k+n_j} \leq 1$.
 \end{proof}
 \begin{lemma}
 Assume that $T_k^{r}, T_k^{t}, \Lambda_k, \Theta_k$ are the temperatures generated by solutions of \eqref{BGK} coupled with \eqref{kin_Temp} and \eqref{internal}. Assume that $l_k \geq 1$ and that
\begin{align*}
T_k^{r} \geq c_k \Theta_k\quad \text{with} \quad c_k= \frac{1}{l_k}(d+l_k)^2 \max \left\lbrace 1, \frac{A(t)}{B(t)} \right\rbrace. 
\end{align*}
Moreover, assume $ \frac{\nu_{kk} n_k}{z_k} \geq (\widetilde{\nu}_{kk} + \widetilde{\nu}_{kj}) \frac{A(T)}{B(T)},$
then the following inequality is satisfied. 
\begin{align*}
 \nu_{kk} n_k \int \int \ln M_k (\widetilde{M}_k - M_k) dv d\eta_{l_k} + 3 z_k \int \int \nu_{kj} n_k \ln M_k (M_k - f_k) dv d\eta_{l_k} \\+ 3 z_k \int \int \nu_{kj} n_j \ln M_k (M_{kj} -f_k) dv \eta_{l_k} \leq 0 
 \end{align*}
 \label{lemm3}
 \end{lemma}
 \begin{proof}
 By using that $\ln M_k = \ln( \frac{n_k}{\sqrt{2 \pi \Lambda_k}^d} \frac{1}{\sqrt{2 \pi \Theta_k}^{l_k}}) - \frac{|v-u_k|^2}{2 \Lambda_k/m_k} - \frac{|\eta_{l_k}- \bar{\eta}_{l_k} |^2}{2 \Theta_k/m_k}$ and the definitions of the macroscopic quantities of $f_k$, $M_k$, $\widetilde{M}_k$ and $M_{kj}$, we obtain that the integrals are given by
 \begin{align*}
- \nu_{kk} n_k \left(\frac{d}{2} \frac{T_k}{\Lambda_k} + \frac{l_k}{2} \frac{T_k}{\Theta_k} - \frac{d+l_k}{2} \right) + 3 z_k \nu_{kk} n_k \left(\frac{d}{2} \frac{T_k^{t}}{\Lambda_k} + \frac{l_k}{2} \frac{T_k^{r}}{\Theta_k}\right) \\+ 3 z_k \nu_{kj} n_j \left( \frac{d}{2} \frac{T_k^{t}}{\Lambda_k} + \frac{l_k}{2} \frac{T_k^{r}}{\Theta_k}\right) - 3 z_k \nu_{kk} n_k \frac{d+l_k}{2} - 3 z_k \nu_{kj} n_j \left( \frac{d}{2} \frac{\Lambda_{kj}}{\Lambda_k} + \frac{l_k}{2} \frac{\Theta_{kj}}{\Theta_k} \right)
 \end{align*}
 This term is non-negative if $z_k$ satisfies \eqref{inz}. The transformation to \eqref{inz} is possible since both the nominator and the denominator are positive according to lemma \ref{lemm1}. The inequality \eqref{inz} is true according to lemma \ref{lemm2}. 
 \end{proof}
\begin{theorem}
\label{conv1}
 Assume that $T_k^{r}, T_k^{t}, \Lambda_k, \Theta_k$ are the temperatures generated by solutions of \eqref{BGK} coupled with \eqref{kin_Temp} and \eqref{internal}. Assume $\nu_{kk} n_k \geq \nu_{kj} n_j + c \nu_{kk} n_k ,$ $0<c <1$, $ k,j=1,2, k \neq j,$ and assume that $l_k \geq 1$ and that
\begin{align*}
T_k^{r} \geq c_k \Theta_k\quad \text{with} \quad c_k= \frac{1}{l_k}(d+l_k)^2 \max \left\lbrace 1, \frac{A(t)}{B(t)} \right\rbrace. 
\end{align*}
Moreover, assume $ \frac{\nu_{kk} n_k}{z_k} \geq (\widetilde{\nu}_{kk} + \widetilde{\nu}_{kj}) \frac{A(T)}{B(T)}.$ Then, in the space homogeneous case, we have the following convergence rate of the distribution functions $f_1$ and $f_2$:
\begin{align*}
||f_k - M_k||_{L^1(dv)}  \leq 4 e^{-\frac{1}{2} Ct} \left(\sum_{k=1}^2 \left( H_k(f_k^0|\widetilde{M}_k^0)+ 3 z_k H_k(M_k^0|\widetilde{M}_k^0)\right) \right)^{\frac{1}{2}}.
\end{align*}
where $C$ is given by $$C= \text{min} \left\lbrace  \nu_{11} n_1 +  \nu_{12} n_2 , \nu_{22} n_2 + \nu_{21} n_1, \frac{2c}{3 z_1} \nu_{11} n_1, \frac{2c}{3z_2} \nu_{22} n_2 \right\rbrace.$$
\end{theorem}

\begin{proof}
We consider the entropy production of species $1$ defined by
\begin{align*}
D_1(f_1, f_2) = &- \int \int  \nu_{11} n_1 \ln f_1 ~ (M_1 - f_1) dv d\eta_{l_1} \\&- \int \int  \nu_{12} n_2 \ln f_1 ~(M_{12} - f_1) dv d\eta_{l_1} \\&- 3 \int \int \nu_{11} n_1 (\widetilde{M}_1 - M_1) \ln M_1 dv d\eta_{l_1}\\&- 3 z_1 \int \int \nu_{11} n_1 (M_1-f_1) \ln M_1 dv d\eta_{l_1} \\&- 3 z_1 \int \int \nu_{12} n_2 (M_{12} - f_1) \ln M_1 dv d\eta_{l_1}.
\end{align*}
Define the function  $h(x) := x \ln x-x$. The function satisfies  $h'(x)= \ln x$, so we can deduce 
\begin{align*}
D_1(f_1, f_2) = &- \int \int  \nu_{11} n_1 h'(f_1) (M_1 - f_1) dv d\eta_{l_1} \\ &- \int \int \nu_{12} n_2 h'(f_1) (M_{12} - f_1) dv d\eta_{l_1}\\&- 3 \int \int \nu_{11} n_1 (\widetilde{M}_1 - M_1) h'(M_1) dv d\eta_{l_1} \\&- 3 z_1 \int \int \nu_{12} n_1 h'(M_1) (M_1 -f_1) dv d\eta_{l_1} \\&- 3 z_1 \int \int \nu_{12} n_2 h'(M_1) (M_{12} -f_1) dv d\eta_{l_1}.
\end{align*}
%since $\int (f_i-M_i)dv=\int (f_i-M_{ie})dv=0$. 
Since $h$ is convex, we obtain
\begin{align*}
\begin{split}
D_1(f_1, f_2) & \geq \int \int \nu_{11} n_1 (h(f_1) - h(M_1)) dv d\eta_{l_1} \\&+ \int \int  \nu_{12} n_2 ( h(f_1) - h(M_{12})) dv d\eta_{l_1} \\&+ 2  \int \int \nu_{11} n_1(h(M_1) - h(\widetilde{M}_1)) dv d\eta_{l_1}\\&- \int \int \nu_{11} n_1 (\widetilde{M}_1 - M_1) h'(M_1) dv d\eta_{l_1} \\&- 3 z_1 \int \int \nu_{12} n_1 h'(M_1) (M_1 -f_1) dv d\eta_{l_1} \\&- 3 z_1 \int \int \nu_{12} n_2 h'(M_1) (M_{12} -f_1) dv d\eta_{l_1}.
 \end{split}
\end{align*}
According to lemma \ref{lemm3}, the last three terms are positive and we obtain
\begin{align}
\begin{split}
D_1(f_1, f_2) & \geq \int \int \nu_{11} n_1 (h(f_1) - h(M_1)) dv d\eta_{l_1} \\&+ \int \int  \nu_{12} n_2 ( h(f_1) - h(M_{12})) dv d\eta_{l_1} \\&+ 2  \int \int \nu_{11} n_1(h(M_1) - h(\widetilde{M}_1)) dv d\eta_{l_1}\\&= 
 \nu_{11} n_1 (H(f_1) - H(M_1)) +  \nu_{12} n_2 ( H(f_1) - H(M_{12})) \\&+ 2 \nu_{11} n_1 (H(M_1)- H(\widetilde{M}_1)).
 \end{split}
 \label{eq:D}
\end{align}
In the same way we get a similar expression for $D_2(f_2,f_1)$  just exchanging the indices $1$ and $2$. \\ 
%If we use that $\ln M_i$ is a linear combination of $1,v$ and $|v|^2$, we see that $\int (M_i-f_i) \ln M_i dv=0$ since $f_i$ and $M_i$ have the same moments. With this we can compute that 
%\begin{align}
%H(f_i | M_i) = H(f_i)- H(M_i).
%\label{eq:relentr}
%\end{align}
According to lemma 3.5 in \cite{Pirner5}, we see that
\begin{align}
 \nu_{12} n_2 H(M_{12}) + \nu_{21} n_1 H(M_{21}) \leq \nu_{12} n_2 H(M_1) + \nu_{12} n_1 H(M_2)
 \label{eq:theo2.7}.
\end{align}
With  \eqref{eq:theo2.7}, we can deduce from \eqref{eq:D} that
\begin{align*}
\begin{split}
&D_1(f_1,f_2)+ D_2(f_2,f_1) \geq \left(\nu_{11} n_1 +\nu_{12} n_2 \right)( H(f_1)-H(M_1))\\ &+ \left(\nu_{22} n_2 +\nu_{21} n_1 \right) (H(f_2)-H(M_2)) + 2 \sum_{k=1}^{2} \nu_{kk} n_k (H(M_k)-H(\widetilde{M}_k)).
\end{split}
\end{align*}
According to  lemma \ref{Mtilde}, the last term on the right- hand side is non-negative. Therefore, we can use the assumption that $\nu_{kk} n_k \geq \nu_{kj} n_j+c \nu_{kk} n_k ,~ k,j =1,2, k \neq j,$ and get
\begin{align}
\begin{split}
D_1(f_1,f_2)&+D_2(f_2,f_1) \geq (\nu_{11} n_1 + \nu_{12} n_2) (H(f_1) - H(\widetilde{M}_1))\\& + (\nu_{22}n_2 + \nu_{21} n_1) (H(f_2) - H(\widetilde{M}_2))+c \sum_{k=1}^2 \nu_{kk} n_k (H(M_k)-H(\widetilde{M}_k))
 \end{split}
\label{prod}
\end{align}
Now, we want to consider the time derivative of the relative entropies
\begin{align*}
H_k(f_k|\widetilde{M}_k)+ 3 z_k H_k(M_k|\widetilde{M}_k)&= H(f_k)-H(\widetilde{M}_k) +3z_k( H(M_k)- H(\widetilde{M}_k))\\& = \int \int f_k \ln \frac{f_k}{\widetilde{M}_k} dv d\eta_{l_k} + 3 z_k \int \int M_k \ln \frac{M_k}{\widetilde{M}_k} dv d\eta_{l_k}
\end{align*}
for $k=1,2$. The last equality follows from the fact that $f_k$, $M_k$ and $\widetilde{M}_k$ have the same densities, mean velocities and internal energies.  We want to relate these functions
 to the entropy production in the following. First, we use product rule and obtain
\begin{align}
\begin{split}
&\frac{d}{dt}  \left( \sum_{k=1}^2 \left( \int \int f_k \ln \frac{f_k}{\widetilde{M}_k} dv d\eta_{l_k} + \int \int M_k \ln \frac{M_k}{\widetilde{M}_k} dv d\eta_{l_k} \right) \right)\\ &= \int \int \left( \partial_t f_1 \ln f_1  + \partial_t f_1 - \partial_t f_1 \ln \widetilde{M}_1 - \frac{f_1}{\widetilde{M}_1} \partial_t \widetilde{M}_1 \right) dv d\eta_{l_1}\\& +\int \int \left( \partial_t f_2 \ln f_2  + \partial_t f_2 - \partial_t f_2 \ln \widetilde{M}_2 -  \frac{f_2}{\widetilde{M}_2}\partial_t \widetilde{M}_2 \right) dv d\eta_{l_2}\\ &+ \int \int \left( \partial_t M_1 \ln M_1  + \partial_t M_1 - \partial_t M_1 \ln \widetilde{M}_1 - \frac{M_1}{\widetilde{M}_1} \partial_t \widetilde{M}_1 \right) dv d\eta_{l_1}\\& +\int \int \left( \partial_t M_2 \ln M_2  + \partial_t M_2 - \partial_t M_2 \ln \widetilde{M}_2 -  \frac{M_2}{\widetilde{M}_2}\partial_t \widetilde{M}_2 \right) dv d\eta_{l_2}
\end{split}
\label{eq:derentr}
\end{align}
The terms  $\int \int \partial_t f_k dv d\eta_{l_k}$ and $\int \int \partial_t M_k dv d\eta_{l_k}$ vanish since the densities are constant in the space homogeneous case.
By using the explicit expression of $\partial_t \widetilde{M}_k$ given by 
{\footnotesize
\begin{align}
\begin{split}
\partial_t \widetilde{M}_k &= \Big(\frac{\partial_t n_k}{n_k} + \frac{v-u_k}{T_k /m_k} \cdot \partial_t u_k+ \frac{\eta_{l_k}- \bar{\eta}_{l_k}}{T_k/m_k} \cdot \partial_t \bar{\eta}_{l_k}\\&+ \left(\frac{|v-u_k|^2+|\eta_{l_k} - \bar{\eta}_{l_k}|^2}{2 (T_k/m_k)^2}-\frac{l_k+d}{2 T_k/m_k}\right)\partial_t  T_k/m_k \Big) M_k,
\end{split}
 \label{partialMtilde}
\end{align} }
we can compute by using that $f_k$, $M_k$ and $\widetilde{M}_k$ have the same densities, mean velocities and internal energies that $$\int f_k \frac{\partial_t \widetilde{M}_k}{\widetilde{M}_k} dvd\eta_{l_k}=\int M_k \frac{\partial_t \widetilde{M}_k}{\widetilde{M}_k} dvd\eta_{l_k} =\partial_t n_k =0,\quad k=1,2,$$ since $n_k$ is constant in the space-homogeneous case. 
On the right-hand side of \eqref{eq:derentr}, we insert $\partial_t f_1$ and $\partial_t f_2$, and $\partial_t M_1$ and $\partial_t M_2$ from equation \eqref{BGK} and and obtain 
\begin{align*}
\frac{d}{dt} &\left(\sum_{k=1}^2 \left( H_k(f_k|\widetilde{M}_k)+ 3 z_k H_k(M_k|\widetilde{M}_k)\right) \right)\\&=\int \int \left(  \nu_{11} n_1 (M_1 - f_1) + \nu_{12} n_2 (M_{12} - f_1) \right) \ln f_1 dv d\eta_{l_1} \\&+ \int \int \left( \nu_{22} n_2 (M_2 - f_2) +  \nu_{21} n_2 (M_{21} - f_2) \right) \ln f_2 dv d\eta_{l_2} \\&+3 \int \int ( \nu_{11} n_1 (\widetilde{M}_1 - M_1) + \nu_{11} n_1 (M_1 - f_1) + \nu_{12} n_2 (M_{12} - f_1)) \ln M_1 dv d\eta_{l_1} \\ &+3 \int \int ( \nu_{22} n_2 (\widetilde{M}_2 - M_2) + \nu_{22} n_2 (M_2 - f_2) + \nu_{21} n_1 (M_{21} - f_2)) \ln M_2 dv d\eta_{l_2}.
\end{align*}
Indeed, the terms with $\ln \widetilde{M}_1$ and $\ln \widetilde{M}_2$ vanish since $\ln \tilde{M}_1$ and $\ln \tilde{M}_2$ are a linear combination of $1,v$ and $|v|^2+ |\eta_{l_k}|^2$ and our model satisfies the conservation of the number of particles, total momentum and total energy (see section 3.1 in \cite{Pirner5}). All in all, we obtain
\begin{align}
\begin{split}
\frac{d}{dt} \left(\sum_{k=1}^2 \left( H_k(f_k|\widetilde{M}_k)+ 3 z_k H_k(M_k|\widetilde{M}_k)\right) \right) = - (D_1(f_1,f_2)+ D_2(f_2,f_1)).
\end{split}
\end{align}
 Using \eqref{prod} we obtain
 {\scriptsize
\begin{align*}
\frac{d}{dt} &\left(\sum_{k=1}^2 \left( H_k(f_k|\widetilde{M}_k)+ 3 z_k H_k(M_k|\widetilde{M}_k)\right) \right)\\ &\leq -\left[\left(\nu_{11} n_1 + \nu_{12} n_2 \right) H(f_1|\widetilde{M}_1) + \left(\nu_{22} n_2 + \nu_{21} n_1 \right) H(f_2 |\widetilde{M}_2)+ c \sum_{k=1}^2 \nu_{kk} n_k H(M_k|\widetilde{M}_k)\right]  \\ &\leq - \text{min} \left\lbrace  \nu_{11} n_1 +  \nu_{12} n_2 , \nu_{22} n_2 + \nu_{21} n_1, \frac{c}{3 z_1} \nu_{11} n_1 , \frac{c}{3z_2} \nu_{22} n_2 \right\rbrace \\  &( H( f_1| \widetilde{M}_1) + H(f_2|\widetilde{M}_2) + 3  (z_1 H_1(M_1|\widetilde{M}_1)+z_2 H_2(M_2|\widetilde{M}_2)) \\ &\leq - \frac{c}{3} \text{min} \left\lbrace  \nu_{11} n_1 +  \nu_{12} n_2 , \nu_{22} n_2 + \nu_{21} n_1, \frac{1}{ z_1} \nu_{11} n_1, \frac{1}{z_2} \nu_{22} n_2 \right\rbrace \\&( H( f_1| \widetilde{M}_1) + H(f_2|\widetilde{M}_2) + 3 z_k (H_1(M_1|\widetilde{M}_1)+ H_2(M_2|\widetilde{M}_2).
\end{align*}}
Define $C:= \text{min} \left\lbrace  \nu_{11} n_1 +  \nu_{12} n_2 , \nu_{22} n_2 + \nu_{21} n_1, \frac{1}{ z_1} \nu_{11} n_1, \frac{1}{z_2} \nu_{22} n_2 \right\rbrace ,$ then we can deduce an exponential decay with Gronwall's inequality
\begin{align*}
H(f_k|\widetilde{M}_k) &\leq \left(\sum_{k=1}^2 \left( H_k(f_k|\widetilde{M}_k)+ 3 z_k H_k(M_k|\widetilde{M}_k)\right) \right)\\& \leq e^{-\frac{2c}{3} Ct} \left(\sum_{k=1}^2 \left( H_k(f_k^0|\widetilde{M}_k^0)+ 3 z_k H_k(M_k^0|\widetilde{M}_k^0)\right) \right), \quad k=1,2.
\end{align*}
With the Ciszar-Kullback inequality (see \cite{Matthes}) we get
\begin{align*}
||f_k - M_k||_{L^1(dv)} & \leq 4 (H(f_k|\widetilde{M}_k))^{\frac{1}{2}} \\ & \leq 4 e^{-\frac{c}{3} Ct} [\left(\sum_{k=1}^2 \left( H_k(f_k^0|\widetilde{M}_k^0)+ 3 z_k H_k(M_k^0|\widetilde{M}_k^0)\right) \right)]^{\frac{1}{2}}.
\end{align*}
\end{proof}
%\subsubsection{Conclusion and Comparison with the BGK model for polyatomic molecules of Andries, Le Tallec, Perlat and Perthame}

\begin{remark}
\label{rem2}
If we look at the proof of theorem \ref{conv1} and the proof of the H-theorem, we see that the restriction on $z_k$ appears due to the additional terms $\nu_{kk} n_k (M_{kk}-f_k)$ and $\nu_{kj} n_j (M_{kj}-f_k)$. These terms were added in the extension to gas mixtures in order to ensure that the time evolution of the mean velocity and the internal energy are the same using \eqref{BGK} or \eqref{kin_Temp}. The original model for one species presented in \cite{Bernard} and section \ref{2.2} does not have this term.
\end{remark}
Now, the question is if we can find a BGK model for gas mixtures which captures both regimes, slow and fast relaxation of the temperatures, and leads to a reasonable convergence rate to equilibrium. Remark \ref{rem2} indicates that for one species the model in section \ref{2.2} does satisfy both of this required properties, only the extension to gas mixtures fails, because of these additional terms. This motivates us to another extension to gas mixtures which satisfies conservation properties, but also an H-theorem with less assumptions than the model in \cite{Pirner5}, which covers both regimes and produces a reasonable convergence rate to equilibrium. This is presented in the next section.

\section{A new model for gas mixtures which captures both slow and fast relaxation of the temperatures}
\label{sec2}
Again we consider the following two kinetic equations as in section \ref{sec2.2}.
\begin{align} \begin{split} \label{BGK3}
\partial_t f_1 + v\cdot\nabla_x   f_1   &= \nu_{11} n_1 (M_1 - f_1) + \nu_{12} n_2 (M_{12}- f_1),
\\ 
\partial_t f_2 + v\cdot\nabla_x   f_2 &=\nu_{22} n_2 (M_2 - f_2) + \nu_{21} n_1 (M_{21}- f_2), \\
f_1(t=0) &= f_1^0, \\
f_2(t=0) &= f_2^0
\end{split}
\end{align}
with the Maxwell distributions
\begin{align} 
\begin{split}
M_k(x,v,\eta_{l_k},t) &= \frac{n_k}{\sqrt{2 \pi \frac{\Lambda_k}{m_k}}^d } \frac{1}{\sqrt{2 \pi \frac{\Theta_k}{m_k}}^{l_k}} \exp({- \frac{|v-u_k|^2}{2 \frac{\Lambda_k}{m_k}}}- \frac{|\eta_{l_k}- \bar{\eta}_{l_k}|^2}{2 \frac{\Theta_k}{m_k}}), 
\\
%M_2(x,v,\eta_{l_2},t) &= \frac{n_2}{\sqrt{2 \pi \frac{\Lambda_2}{m_2}}^d } \frac{1}{\sqrt{2 \pi \frac{\Theta_2}{m_2}}^{l_2}} \exp({- \frac{|v-u_2|^2}{2 \frac{\Lambda_2}{m_2}}}- \frac{|\eta_{l_2}|^2}{2 \frac{\Theta_2}{m_2}})
%\\
M_{kj}(x,v,\eta_{l_k},t) &= \frac{n_{kj}}{\sqrt{2 \pi \frac{\Lambda_{kj}}{m_k}}^d } \frac{1}{\sqrt{2 \pi \frac{\Theta_{kj}}{m_k}}^{l_k}} \exp({- \frac{|v-u_{kj}|^2}{2 \frac{\Lambda_{kj}}{m_k}}}- \frac{|\eta_{l_k}- \bar{\eta}_{l_k,kj}|^2}{2 \frac{\Theta_{kj}}{m_k}}), 
%\\
%M_{21}(x,v,\eta_{l_2},t) &= \frac{n_{21}}{\sqrt{2 \pi \frac{\Lambda_{21}}{m_2}}^d } \frac{1}{\sqrt{2 \pi \frac{\Theta_{21}}{m_2}}^{l_2}} \exp({- \frac{|v-u_{21}|^2}{2 \frac{\Lambda_{21}}{m_2}}}- \frac{|\eta_{l_2}|^2}{2 \frac{\Theta_{21}}{m_2}})
\end{split}
\label{BGKmix3}
\end{align}
for $ j,k =1,2, j \neq k$ with the condition 
\begin{equation} 
\nu_{12}=\varepsilon \nu_{21}, \quad 0 < \frac{l_1}{l_1+l_2}\varepsilon \leq 1.
\label{coll3}
\end{equation}
Again the equation is coupled with conservation of internal energy 
\eqref{internal},
but we replace the relaxation equation \eqref{kin_Temp} by the following modified version
%\begin{align}
%\partial_t M_k + v \cdot \nabla_x M_k = \frac{\nu_{kk} n_k}{Z_r^k} \frac{d+l_k}{d} (\tilde{M}_k - M_k), \quad k=1,2
%\label{kin_Temp}
%\end{align}
%\textcolor{cyan}{New version: \begin{align}
% \partial_t M_k + v \cdot \nabla_x M_k = \frac{\nu_{kk} n_k}{Z_r^k} \frac{d+l_k}{d} (\tilde{M}_k - M_k)+ \frac{\nu_{kj} n_j}{Z_r^k} \frac{d+l_k}{d} (\tilde{M}_{kj} - M_k)+ \nu_{kj} n_j (M_{kj} - M_k) , \quad k=1,2 \end{align}
% }
% \textcolor{red}{
 \begin{align}
 \begin{split}
 \partial_t M_k + v \cdot \nabla_x M_k = \frac{\nu_{kk} n_k}{Z_r^k} \frac{d+l_k}{d} (\widetilde{M}_k - M_k)+ \nu_{kj} n_j  (\widetilde{M}_{kj} - M_k) , \\
 \Theta_k(0)= \Theta_k^0
 \end{split} 
 \label{kin_Temp3}
 \end{align}
 %}
for $j,k=1,2, j \neq k$.
 Here $M_k$ is still defined by \eqref{Maxwellian}
and $\widetilde{M}_k$ by \eqref{Max_equ}.  The additional $\widetilde{M}_{kj}$ is defined by
\begin{align}
\widetilde{M}_{kj}= \frac{n_k}{\sqrt{2 \pi \frac{T_{kj}}{m_k}}^{d+l_k}} \exp \left(- \frac{m_k |v-u_{kj}|^2}{2 T_{kj}}- \frac{m_k|\eta_{l_k}- \bar{\eta}_{l_k,kj}|^2}{2 T_{kj}} \right), \quad k=1,2.
\label{Max_equ3}
\end{align}
where $T_{kj}$ is  given by 
\begin{align}
T_{kj}:= \frac{d \Lambda_{kj} + l_k \Theta_{kj}}{d+l_k}.
\label{equ_temp3}
\end{align}
Now, this means that we use equation \eqref{kin_Temp3} instead of \eqref{kin_Temp}  to involve the temperature $\Theta_k$. If we multiply \eqref{kin_Temp3} by $|\eta_{l_k}|^2$, integrate with respect to $v$ and $\eta_{l_k}$ and use \eqref{equ_temp}, we now obtain  
\begin{align}
\begin{split}
\partial_t(n_k \Theta_k) +   \nabla_x\cdot (n_k \Theta_k u_k) = \frac{\nu_{kk} n_k}{Z_r^k} n_k (\Lambda_k - \Theta_k)+ \nu_{kj} n_j n_k(T_{kj} - \Theta_k).
\end{split}
\label{relax3}
\end{align} for $k=1,2$.  $\Lambda_k$ is still determined using equation \eqref{internal}. 

In addition,  \eqref{BGK} and \eqref{kin_Temp} are still consistent. If we multiply the equations for species $k$ of \eqref{BGK} and \eqref{kin_Temp}  by $v$ and integrate with respect to $v$ and $\eta_{l_k}$, we get in both cases for the right-hand side 
$$ \nu_{kj} n_j n_k  (u_{jk} - u_k),$$ and if we compute the total internal energy of both equations, we obtain in both cases $$ \frac{1}{2} \nu_{kj} n_k n_j [d \Lambda_{jk} + l_j \Theta_{jk} - ( d \Lambda_j + l_j \Theta_j)],$$ by the use of \eqref{internal}.

Since, we did not change equation \eqref{BGKmix}, we still have conservation of the number of particles if we assume \eqref{density2}, conservation of total momentum if we assume \eqref{convexvel2}, \eqref{veloc2} and conservation of total energy if we assume \eqref{contemp2} and \eqref{temp2}.

The existence and uniqueness of mild solutions of this modified model is proven in \cite{Pirner8}.
%, since one can handle the new term $\frac{\nu_{kj} n_j}{Z_r^k} \frac{d+l_k}{d} (\widetilde{M}_{kj} - f_k)$ in the same manner as the term $\frac{\nu_{kk} n_k}{Z_r^k} \frac{d+l_k}{d} (\widetilde{M}_k - M_k)$. Especially, one can prove the additional estimates
%\begin{align*}
% T_{12}(t), T_{21}(t) &\leq A(t) < \infty, \\
%T_{12}(t), T_{21}(t)  &\geq B(t)>0,
%\end{align*} 
%with $A(t),B(t)$ given by
%$$ A(t) = C e^{Ct}, ~ B(t) = C e^{-Ct}, ~ C>0.$$
\subsection{Equilibrium, H-theorem and entropy inequality}
The mew model has the following characterization of equilibrium
\begin{theorem}[Equilibrium]
Assume $f_1, f_2 >0$ with $f_1$ and $f_2$ independent of $x$ and $t$.
Assume the conditions \eqref{density2}, \eqref{convexvel2}, \eqref{veloc2}, \eqref{contemp2} and \eqref{temp2}, $\delta \neq 1, \alpha \neq 1, l_1,l_2\neq 0$, so that all temperatures are positive. %especially \eqref{gamma}

Then $f_1$ and $f_2$ are Maxwell distributions with equal mean velocities $u_1=u_2=u_{12}=u_{21}$ and temperatures $T:=T_1^{r}=T_2^{r}=T_1^{t}=T_2^{t}=\Lambda_1=\Lambda_2=\Theta_1=\Theta_2=\Theta_{12}=\Theta_{21}= \Lambda_{12}=\Lambda_{21}$. This means $f_k$ is given by
$$ M_k(x,v,\eta_{l_k},t) = \frac{n_k}{\sqrt{2 \pi \frac{T}{m_k}}^d } \frac{1}{\sqrt{2 \pi \frac{T}{m_k}}^{l_k}} \exp({- \frac{|v-u|^2}{2 \frac{T}{m_k}}}- \frac{|\eta_{l_k}|^2}{2 \frac{T}{m_k}}), \quad k=1,2.$$
\label{equilibrium3}
\end{theorem}
\begin{proof}
%If $Q_{11}(f_1,f_1)+Q_{12}(f_1,f_2)=0$ and $Q_{22}(f_2,f_2)+Q_{21}(f_2,f_1)=0$ , then $\ln f_1 ~ Q_{11}(f_1,f_1)+\ln f_1 ~ Q_{12}(f_1,f_2)+\ln f_2 ~ Q_{22}(f_2,f_2)+\ln f_2 ~ Q_{21}(f_2,f_1)=0$ and so we have equality in the H-theorem.

Equilibrium means that $f_1,f_2, \Lambda_1, \Lambda_2, \Theta_1, \Theta_2$ are independent of $x$ and $t$. 
 Thus in equilibrium the right-hand side of the equations \eqref{BGK3}  and \eqref{kin_Temp3} have to be zero. From the right-hand side of equation \eqref{kin_Temp3}, we obtain 
\begin{align}
\frac{\nu_{kk} n_k}{Z_r^k}  (\widetilde{M}_k - M_k)+ \nu_{kj} n_j \frac{d+l_k}{d} (\widetilde{M}_{kj} - M_k) =0
\label{A}
\end{align}
 If we compute the internal energy of this expression, we obtain
\begin{align}
l_k \Theta_{kj} + d \Lambda_{kj} = d \Lambda_j + l_k \Theta_k
\label{B}
\end{align}
If we now compute the rotational and vibrational temperature of \eqref{A}, and use \eqref{B}, we obtain
\begin{align}
\Lambda_k = \Theta_k
\label{C}
\end{align}
For the other equalities one has to prove, we only use \eqref{C} and the equation \eqref{BGK3}. But since the equation \eqref{BGK3} has not changed compared to the old model in section \ref{sec2.2}, the rest of the proof is exactly the same as in the proof of theorem \ref{equilibrium}.
\end{proof}
Next, we want to prove the H-theorem for the modified model. 
%Define $\frac{1}{z_k} := \frac{1}{Z_k^r} \frac{d+l_k}{d}, ~ k=1,2$. Then, one can prove the following H-theorem. 
\begin{theorem}[H-theorem for mixture]
Assume $f_1, f_2 >0$. Assume %$z_1,z_2 \leq 1$ and 
 $\alpha, \delta\neq 1, l_1,l_2 \neq 0.$
Assume the relationship between the collision frequencies \eqref{coll}, the conditions for the interspecies Maxwellians \eqref{density2}, \eqref{convexvel2}, \eqref{veloc2}, \eqref{contemp2} and \eqref{temp2} and the positivity of all temperatures, then
\begin{align*}
\sum_{k=1}^2 &[ \nu_{kk} n_k \int \int(M_k - f_k) \ln f_k dv d\eta_{l_k} \\&+ \frac{
\nu_{kk} n_k \max\lbrace 1, z_1, z_2 \rbrace}{z_k}
 \int \int(\widetilde{M}_k - M_k) \ln M_k dv d\eta_{l_k}] \\&+ %\frac{
 \nu_{12} n_1 \max\lbrace 1, z_1, z_2 \rbrace
  \int \int (\widetilde{M}_{12} - M_1) \ln M_1 dv d\eta_{l_1} \\&+ %\frac{
  \nu_{21} n_2 \max\lbrace 1,z_1, z_2 \rbrace
  \int \int(\widetilde{M}_{21} - M_2) \ln M_2 dv d\eta_{l_2} \\ &+\nu_{12} n_2 \int \int(M_{12} - f_1) \ln f_1 dv d\eta_{l_1} + \nu_{21} n_1 \int \int (M_{21} - f_2) \ln f_2 dv d\eta_{l_2}  \leq 0,
\end{align*}
with equality if and only if $f_1$ and $f_2$ are Maxwell distributions with equal mean velocities and all temperatures coincide. 
\label{H-theorem3}
\end{theorem}
\begin{proof}
The fact that $ \nu_{kk} n_k \int \int(M_k - f_k) \ln f_k dv d\eta_{l_k} + \frac{
\nu_{kk} n_k \max\lbrace 1, z_1, z_2 \rbrace}{z_k}
\int \int(\widetilde{M}_k - M_k) \ln M_k dv d\eta_{l_k} \leq 0, k=1,2$ can be proven in the same way as lemma \ref{one_species}. In the second step of the proof of lemma \ref{one_species}, we can use the first inequality of lemma 3.1 and therefore $\frac{
 \max\lbrace 1, z_1, z_2 \rbrace}{z_k}$ can be estimated from below by $1$.
 
 It remains to prove that the rest is non-positive.
Let us define 
\begin{align*}
S :=  \nu_{12} n_2%}{z_1}
 \max\lbrace 1, z_1, z_2 \rbrace \int \int(\widetilde{M}_{12} - M_1) \ln M_1 dv d\eta_{l_1}\\ + %\frac{
  \nu_{21} n_1%}{z_2} 
 \max\lbrace 1, z_1, z_2 \rbrace \int \int(\widetilde{M}_{21} - M_2) \ln M_2 dv d\eta_{l_2} \\ +\nu_{12} n_2 \int \int(M_{12} - f_1) \ln f_1 dv d\eta_{l_1} + \nu_{21} n_1 \int \int(M_{21} - f_2) \ln f_2 dv d\eta_{l_2}
\end{align*}
The task is to prove that $S\leq 0$.
Since the function $h(x)= x \ln x -x$ is strictly convex for $x>0$, we have $h'(f) (g-f) \leq h(g) - h(f)$ with equality if and only if $g=f$. So \begin{align}
(g-f) \ln f  \leq g \ln g - f \ln f +f -g 
\label{convex}
\end{align}
Consider now $S$ and apply the inequality \eqref{convex} to each of the terms in $S$.
% In the first term we choose $f=f_1$ and $g=M_{12}$ and in the second term $f=f_2$ and $g=M_{21}$
{\footnotesize
\begin{align*}
\nu_{12} n_2 %}{z_1} 
\max\lbrace 1, z_1, z_2 \rbrace [ \int \int \widetilde{M}_{12} \ln \widetilde{M}_{12}  dv d\eta_{l_1} - \int \int M_1 \ln M_1 dv d\eta_{l_1} \\ + \int \int M_1 dv d\eta_{l_1} - \int \int \widetilde{M}_{12} dv d\eta_{l_1} ]\\ + \nu_{12} n_2 \max\lbrace 1, z_1, z_2 \rbrace[ \int \int \widetilde{M}_{21} \ln \widetilde{M}_{21} dv d\eta_{l_2} - \int \int M_2 \ln M_2 dv d\eta_{l_2}\\+ \int \int M_2 dv\eta - \int \int \widetilde{M}_{21} dv d\eta_{l_2} ]\\+ \nu_{12} n_2 [ \int \int M_{12} \ln M_{12} dv d\eta_{l_1} - \int \int f_1 \ln f_1 dv d\eta_{l_1} + \int \int f_1 dv d\eta_{l_1} - \int \int M_{12} dv d\eta_{l_1} ]\\  + %\frac{
\nu_{21} n_1%}{z_2}
 [ \int \int M_{21} \ln M_{21} dv d\eta_{l_2} - \int \int f_2 \ln f_2 dv d\eta_{l_2} + \int f_2 dv d\eta_{l_2} - \int M_{21} dv d\eta_{l_2}],
\end{align*}}
with equality if and only if $\widetilde{M}_{kj}= M_k$ and $M_{kj}=f_k$. By computing the macroscopic quantities of this equation, one observes that equality means that $f_1$ and $f_2$ are Maxwell distributions with equal mean velocities and all temperatures coincide. % Then $u_{12}= \delta u_1 + (1- \delta) u_2 = u_1$ from which we can deduce $u_1=u_2=u_{21}=u_{12}$ and $\Lambda_1=T_2=T_{12}=T_{21}$. This means $f_1$ and $f_2$ are Maxwell distributions with equal bulk velocity and temperature. \\

Since $\widetilde{M}_{kj}$, $M_{kj}$, $f_k$ and $M_k$ have the same density, the right-hand side reduces to
\begin{align}
\begin{split}
\nu_{12} n_2 %}{z_1} 
\max\lbrace 1, z_1, z_2 \rbrace[ \int \int \widetilde{M}_{12} \ln \widetilde{M}_{12}  dv d\eta_{l_1} - \int \int M_1 \ln M_1 dv d\eta_{l_1}  ]\\ + \nu_{12} n_2 \max\lbrace 1, z_1, z_2 \rbrace[\int \int \widetilde{M}_{21} \ln \widetilde{M}_{21} dv d\eta_{l_2} - \int \int M_2 \ln M_2 dv d\eta_{l_2}]\\+ \nu_{12} n_2 [\int \int M_{12} \ln M_{12} dv d\eta_{l_1} -\int \int f_1 \ln f_1 dv d\eta_{l_1}  ]\\  + %\frac{
\nu_{21} n_1%}{z_2}
 [\int \int M_{21} \ln M_{21} dv d\eta_{l_2} -\int \int f_2 \ln f_2 dv d\eta_{l_2} ],
 \label{In1}
 \end{split}
\end{align}
In lemma 3.5 in \cite{Pirner5} we proved the following inequality
\begin{align}
\begin{split}
\nu_{12} n_2 \int \int M_{12} \ln M_{12} dv d\eta_{l_1}  + \nu_{21} n_1 \int \int M_{21} \ln M_{21} dv d\eta_{l_2} \\ \leq \nu_{12} n_2 \int \int M_1 \ln M_1 dv d\eta_{l_1} + \nu_{21} n_1 \int \int M_2 \ln M_2 dv d\eta_{l_2}.
\end{split}
\label{In2}
\end{align}
In the same way we can also prove
\begin{align}
\begin{split}
\nu_{12} n_2 \int \int \widetilde{M}_{12} \ln \widetilde{M}_{12} dv d\eta_{l_1}  + \nu_{21} n_1 \int  \int \widetilde{M}_{21} \ln \widetilde{M}_{21} dv d\eta_{l_2} \\ \leq \nu_{12} n_2 \int \int \widetilde{M}_1 \ln \widetilde{M}_1 dv d\eta_{l_1} + \nu_{21} n_1 \int \int \widetilde{M}_2 \ln \widetilde{M}_2 dv d\eta_{l_2}.
\end{split}
\label{In3}
\end{align}
With this inequality and lemma \ref{Mtilde}, we see that the first two terms in \eqref{In1} are non-positive. Therefore we can estimate $\max \lbrace 1, z_1, z_2 \rbrace$ from below by $1$. Now, we use \eqref{In2} and \eqref{In3}, All in all, we obtain
\begin{align*}
\begin{split}
\nu_{12} n_2 %}{z_1} 
[ \int \int \widetilde{M}_1 \ln \widetilde{M}_1  dv d\eta_{l_1} - \int \int M_1 \ln M_1 dv d\eta_{l_1}  ]\\ + \nu_{12} n_2 [\int \int \widetilde{M}_{2} \ln \widetilde{M}_{2} dv d\eta_{l_2} - \int \int M_2 \ln M_2 dv d\eta_{l_2}]\\+ \nu_{12} n_2 [\int \int M_{1} \ln M_{1} dv d\eta_{l_1} -\int \int f_1 \ln f_1 dv d\eta_{l_1}  ]\\  + %\frac{
\nu_{21} n_1%}{z_2}
 [\int \int M_{2} \ln M_{2} dv d\eta_{l_2} -\int \int f_2 \ln f_2 dv d\eta_{l_2} ]\\ = \nu_{12} n_2 [\int \int \widetilde{M}_{1} \ln \widetilde{M}_{1} dv d\eta_{l_1} -\int \int f_1 \ln f_1 dv d\eta_{l_1}  ]\\  + %\frac{
\nu_{21} n_1%}{z_2}
 [\int \int \widetilde{M}_{2} \ln \widetilde{M}_{2} dv d\eta_{l_2} -\int \int f_2 \ln f_2 dv d\eta_{l_2} ],
 \end{split}
\end{align*}
That this is non-positive is proven in the proof of theorem \ref{one_species}.
\end{proof}
Define  the
 total entropy \begin{align*}
 \begin{split}
 H(f_1,f_2) = \int (f_1 \ln f_1  + \max \lbrace 1, z_1, z_2 \rbrace M_1 \ln M_1) dv d\eta_{l_1} \\+ \int (f_2 \ln f_2+ \max \lbrace 1, z_1, z_2 \rbrace M_2 \ln M_2 ) dvd\eta_{l_2}.
 \end{split}
% \label{H-funct3}
 \end{align*} 
 Then, with theorem  \ref{H-theorem3}, one can deduce the following corollary.
  \begin{cor}[Entropy inequality for mixtures]
Assume $f_1, f_2 >0$, 
Assume relationship \eqref{coll}, the conditions \eqref{density2}, \eqref{convexvel2}, \eqref{veloc2}, \eqref{contemp2} and \eqref{temp2} and the positivity of all temperatures, then we have the following entropy inequality
\begin{align*}
\partial_t \left(  H(f_1,f_2) \right) \\+ \nabla_x \cdot \big(\int  v (f_1 \ln f_1  + \max \lbrace 1, z_1, z_2 \rbrace M_1 \ln M_1) dv d\eta_{l_1} \\ + \int v( f_2 \ln f_2+ \max \lbrace 1, z_1, z_2 \rbrace M_2 \ln M_2 ) dv d\eta_{l_2} \big) \leq 0,
\end{align*}
with equality if and only if $f_1$ and $f_2$ are Maxwell distribution and all temperatures coincide.
\end{cor}
\begin{remark}
Note that the entropy inequality for the this model has less assumptions than the entropy inequality for the model in the previous section.
\end{remark}
\subsection{Entropy dissipation estimates and convergence rate to equilibrium}
With the new model, the proof of the entropy production estimate and the proof of the convergence rate to equilibrium are much simpler and contain less assumptions.
\begin{theorem}
\label{conv3}
 Assume that $(f_1, f_2, M_1, M_2)$ is a solution of \eqref{BGK3} coupled with \eqref{kin_Temp3} and \eqref{internal}. Then, in the space homogeneous case, we have the following convergence rate of the distribution functions $f_1$ and $f_2$:
 {\small
\begin{align*}
||f_k - \widetilde{M}_k||_{L^1(dv d\eta_{l_k})}  \leq 4 e^{-\frac{1}{4} \widetilde{C} t} \left(\sum_{k=1}^2 \left( H_k(f_k^0|\widetilde{M}_k^0)+ 2 \max \lbrace 1, z_1, z_2 \rbrace H_k(M_k^0|\widetilde{M}_k^0)\right) \right)^{\frac{1}{2}}.
\end{align*}}
where $\widetilde{C}$ is given by $$\widetilde{C}= \min \left\lbrace  \nu_{11} n_1 +  \nu_{12} n_2 , \nu_{22} n_2 + \nu_{21} n_1, \frac{\nu_{11} n_1}{z_1} + \nu_{12} n_2, \frac{\nu_{22} n_2}{z_2} + \nu_{21} n_1 \right\rbrace.$$
\end{theorem}
\begin{proof}
We consider the entropy production of species $1$ defined by
\begin{align*}
D_1(f_1, f_2) = &- \int \int  \nu_{11} n_1 \ln f_1 ~ (M_1 - f_1) dv d\eta_{l_1} \\&- \int \int  \nu_{12} n_2 \ln f_1 ~(M_{12} - f_1) dv d\eta_{l_1} \\&- \frac{\max \lbrace 1, z_1, z_2 \rbrace }{z_1} \int \int \nu_{11} n_1 (\widetilde{M}_1 - M_1) \ln M_1 dv d\eta_{l_1} \\&- \max \lbrace 1, z_1, z_2 \rbrace \int \int \nu_{12} n_2 (\widetilde{M}_{12} - M_1) \ln M_1 dv d\eta_{l_1}.
\end{align*}
Define the function  $h(x) := x \ln x-x$. The function satisfies  $h'(x)= \ln x$, so we can deduce 
\begin{align*}
D_1(f_1, f_2) = &- \int \int  \nu_{11} n_1 h'(f_1) (M_1 - f_1) dv d\eta_{l_1} \\ &- \int \int \nu_{12} n_2 h'(f_1) (M_{12} - f_1) dv d\eta_{l_1}\\&- \frac{\max \lbrace 1, z_1, z_2 \rbrace}{z_1} \int \int \nu_{11} n_1 (\widetilde{M}_1 - M_1) h'(M_1) dv d\eta_{l_1}  \\&- \max \lbrace 1, z_1, z_2 \rbrace \int \int \nu_{12} n_2 h'(M_1) (\widetilde{M}_{12} -M_1) dv d\eta_{l_1}.
\end{align*}
%since $\int (f_i-M_i)dv=\int (f_i-M_{ie})dv=0$. 
Since $h$ is convex, we obtain
\begin{align}
\begin{split}
D_1(f_1, f_2) & \geq \int \int \nu_{11} n_1 (h(f_1) - h(M_1)) dv d\eta_{l_1} \\&+ \int \int  \nu_{12} n_2 ( h(f_1) - h(M_{12})) dv d\eta_{l_1} \\&+  \frac{\max \lbrace 1, z_1, z_2 \rbrace}{z_1} \int \int \nu_{11} n_1(h(M_1) - h(\widetilde{M}_1)) dv d\eta_{l_1} \\&+ \max \lbrace 1, z_1, z_2 \rbrace \int \int \nu_{12} n_2 (h(M_1) - h(\widetilde{M}_{12})) dv d\eta_{l_1}
\\&= 
 \nu_{11} n_1 (H(f_1) - H(M_1)) +  \nu_{12} n_2 ( H(f_1) - H(M_{12})) \\&+ \frac{\max \lbrace 1, z_1, z_2 \rbrace}{z_1} \nu_{11} n_1 (H(M_1)- H(\widetilde{M}_1)) \\&+ \max \lbrace 1, z_1, z_2 \rbrace \nu_{12} n_1 (H(M_1) - H(\widetilde{M}_{12})).
 \end{split}
 \label{eq:D_2}
\end{align}
In the same way we get a similar expression for $D_2(f_2,f_1)$  just exchanging the indices $1$ and $2$. \\ 
%If we use that $\ln M_i$ is a linear combination of $1,v$ and $|v|^2$, we see that $\int (M_i-f_i) \ln M_i dv=0$ since $f_i$ and $M_i$ have the same moments. With this we can compute that 
%\begin{align}
%H(f_i | M_i) = H(f_i)- H(M_i).
%\label{eq:relentr}
%\end{align}
%According to lemma 3.5 in \cite{Pirner5}, we see that
%\begin{align}
% \nu_{12} n_2 H(M_{12}) + \nu_{21} n_1 H(M_{21}) \leq \nu_{12} n_2 H(M_1) + \nu_{12} n_1 H(M_2)
% \label{eq:theo2.7_2}.
%\end{align}
With  \eqref{In2} and \eqref{In3}, we can deduce from \eqref{eq:D_2} that
\begin{align*}
\begin{split}
&D_1(f_1,f_2)+ D_2(f_2,f_1) \geq \left(\nu_{11} n_1 +\nu_{12} n_2 \right)( H(f_1)-H(M_1))\\ &+ \left(\nu_{22} n_2 +\nu_{21} n_1 \right) (H(f_2)-H(M_2)) + \frac{\max \lbrace 1, z_1, z_2 \rbrace}{z_1} \nu_{11} n_1 (H(M_1) -H(\widetilde{M}_1))\\& + \frac{\max \lbrace 1, z_1, z_2 \rbrace}{z_2} \nu_{22} n_2 (H(M_2) - H(\widetilde{M}_2)) + \max \lbrace 1, z_1, z_2 \rbrace \nu_{12} n_2 (H(M_1) -H(\widetilde{M}_1)) \\&+ \max \lbrace 1, z_1, z_2 \rbrace \nu_{21} n_1 (H(M_2) - H(\widetilde{M}_2))
\end{split}
\end{align*}
According to lemma \ref{Mtilde}, we see that $H(M_k) - H(\widetilde{M}_k)$ is non-negative. Therefore, we can estimate $\max \lbrace 1, z_1, z_2 \rbrace$ and $\frac{\max \lbrace 1, z_1, z_2 \rbrace}{z_k}$ by $1$ from below. We obtain
\begin{align}
\begin{split}
D_1(f_1,f_2)&+D_2(f_2,f_1) \geq (\nu_{11} n_1 + \nu_{12} n_2) (H(f_1) - H(\widetilde{M}_1))\\& + (\nu_{22}n_2 + \nu_{21} n_1) (H(f_2) - H(\widetilde{M}_2))
 \end{split}
\label{prod3}
\end{align}
Now, we want to consider the time derivative of the relative entropies
\begin{align*}
H_k(f_k|\widetilde{M}_k)+ \max \lbrace 1, z_1, z_2 \rbrace H_k(M_k|\widetilde{M}_k)&= H(f_k)-H(\widetilde{M}_k) +\max \lbrace 1, z_1, z_2 \rbrace( H(M_k)- H(\widetilde{M}_k))\\& = \int \int f_k \ln \frac{f_k}{\widetilde{M}_k} dv d\eta_{l_k} \\&+ \max \lbrace 1, z_1, z_2 \rbrace \int \int M_k \ln \frac{M_k}{\widetilde{M}_k} dv d\eta_{l_k}
\end{align*}
for $k=1,2$. The last equality follows from the fact that $f_k$ and $\widetilde{M}_k$ have the same densities, mean velocities and internal energies. These functions we want to relate
 to the entropy production in the following. First, we use product rule and obtain
\begin{align}
\begin{split}
&\frac{d}{dt}  \left( \sum_{k=1}^2 \left( \int \int f_k \ln \frac{f_k}{\widetilde{M}_k} dv d\eta_{l_k} + \int \int M_k \ln \frac{M_k}{\widetilde{M}_k} dv d\eta_{l_k} \right) \right)\\ &= \int \int \left( \partial_t f_1 \ln f_1  + \partial_t f_1 - \partial_t f_1 \ln \widetilde{M}_1 - \frac{f_1}{\widetilde{M}_1} \partial_t \widetilde{M}_1 \right) dv d\eta_{l_1}\\& +\int \int \left( \partial_t f_2 \ln f_2  + \partial_t f_2 - \partial_t f_2 \ln \widetilde{M}_2 -  \frac{f_2}{\widetilde{M}_2}\partial_t \widetilde{M}_2 \right) dv d\eta_{l_2}\\ &+ \int \int \left( \partial_t M_1 \ln M_1  + \partial_t M_1 - \partial_t M_1 \ln \widetilde{M}_1 - \frac{M_1}{\widetilde{M}_1} \partial_t \widetilde{M}_1 \right) dv d\eta_{l_1}\\& +\int \int \left( \partial_t M_2 \ln M_2  + \partial_t M_2 - \partial_t M_2 \ln \widetilde{M}_2 -  \frac{M_2}{\widetilde{M}_2}\partial_t \widetilde{M}_2 \right) dv d\eta_{l_2}
\end{split}
\label{eq:derentr_2}
\end{align}
The terms  $\int \int \partial_t f_k dv d\eta_{l_k}$ and $\int \int \partial_t M_k dv d\eta_{l_k}$ vanish since the densities are constant in the space homogeneous case.
By using the explicit expression of $\partial_t \widetilde{M}_k$ given by 
\eqref{partialMtilde},
we can compute by using that $f_k$, $M_k$ and $\widetilde{M}_k$ have the same densities, mean velocities and internal energies that $$\int f_k \frac{\partial_t \widetilde{M}_k}{\widetilde{M}_k} dvd\eta_{l_k}=\int M_k \frac{\partial_t \widetilde{M}_k}{\widetilde{M}_k} dvd\eta_{l_k} =\partial_t n_k =0,\quad k=i,e,$$ since $n_k$ is constant in the space-homogeneous case. 
On the right-hand side of \eqref{eq:derentr_2}, we insert $\partial_t f_1$ and $\partial_t f_2$, and $\partial_t M_1$ and $\partial_t M_2$ from equation \eqref{BGK3}  and \eqref{kin_Temp3}. We obtain 
\begin{align*}
\frac{d}{dt} &\left(\sum_{k=1}^2 \left( H_k(f_k|\widetilde{M}_k)+ \max \lbrace 1, z_1, z_2 \rbrace H_k(M_k|\widetilde{M}_k)\right) \right)\\&=\int \int \left(  \nu_{11} n_1 (M_1 - f_1) + \nu_{12} n_2 (M_{12} - f_1) \right) \ln f_1 dv d\eta_{l_1} \\&+ \int \int \left( \nu_{22} n_2 (M_2 - f_2) +  \nu_{21} n_2 (M_{21} - f_2) \right) \ln f_2 dv d\eta_{l_2} \\&+\frac{\max \lbrace 1, z_1, z_2 \rbrace}{z_1} \int \int ( \nu_{11} n_1 (\widetilde{M}_1 - M_1)  + \max \lbrace 1, z_1, z_2 \rbrace \nu_{12} n_2 (\widetilde{M}_{12} - M_1)) \ln M_1 dv d\eta_{l_1} \\ &+ \int \int ( \frac{\max \lbrace 1, z_1, z_2 \rbrace}{z_2}\nu_{22} n_2 (\widetilde{M}_2 - M_2)  + \max \lbrace 1, z_1, z_2 \rbrace \nu_{21} n_1 (\widetilde{M}_{21} - M_2)) \ln M_2 dv d\eta_{l_2}.
\end{align*}
Indeed, the terms with $\ln \widetilde{M}_1$ and $\ln \widetilde{M}_2$ vanish since $\ln M_1$ and $\ln M_2$ are a linear combination of $1,v$ and $|v|^2+ |\eta_{l_k}|^2$ and our model satisfies the conservation of the number of particles, total momentum and total energy (see section 3.1 in \cite{Pirner5}). All in all, we obtain
\begin{align*}
\begin{split}
\frac{d}{dt} \left(\sum_{k=1}^2 \left( H_k(f_k|\widetilde{M}_k)+\max \lbrace 1, z_1, z_2 \rbrace H_k(M_k|\widetilde{M}_k)\right) \right) = - (D_1(f_1,f_2)+ D_2(f_2,f_1)).
\end{split}
\end{align*}
 Using \eqref{prod3} we obtain
 {\scriptsize
\begin{align}
\begin{split}
\frac{d}{dt} &\left(\sum_{k=1}^2 \left( H_k(f_k|\widetilde{M}_k)+ \max \lbrace 1, z_1, z_2 \rbrace H_k(M_k|\widetilde{M}_k)\right) \right)\\ &\leq -\left[\left(\nu_{11} n_1 + \nu_{12} n_2 \right) H(f_1|\widetilde{M}_1) + \left(\nu_{22} n_2 + \nu_{21} n_1 \right) H(f_2 |\widetilde{M}_2)\right]  \\ &\leq - \text{min} \left\lbrace  \nu_{11} n_1 +  \nu_{12} n_2 , \nu_{22} n_2 + \nu_{21} n_1 \right\rbrace ( H( f_1| \widetilde{M}_1) + H(f_2|\widetilde{M}_2).
\label{add}
\end{split}
\end{align}}
In a similar way, we can compute independently
{\small
\begin{align}
\frac{d}{dt} \left( \sum_{k=1}^2 H_k(M_k|\widetilde{M}_k) \right) \leq - \left[ (\frac{\nu_{11} n_1}{z_1} + \nu_{12} n_2) H(M_1|\widetilde{M}_1) + (\frac{\nu_{22} n_2}{z_2} + \nu_{21} n_1 ) H(M_2|\widetilde{M}_2) \right]
\label{add2}
\end{align}}
Now, we add \eqref{add} and $\max \lbrace 1, z_1, z_2 \rbrace$ times \eqref{add2}. We obtain
 {\scriptsize
\begin{align*}
\begin{split}
\frac{d}{dt} &\left(\sum_{k=1}^2 \left( H_k(f_k|\widetilde{M}_k)+ 2 \max \lbrace 1, z_1, z_2 \rbrace H_k(M_k|\widetilde{M}_k)\right) \right)\\  &\leq - \frac{1}{2} \text{min} \left\lbrace  \nu_{11} n_1 +  \nu_{12} n_2 , \nu_{22} n_2 + \nu_{21} n_1, \frac{\nu_{11} n_1}{z_1} + \nu_{12} n_2, \frac{\nu_{22} n_2}{z_2} + \nu_{21} n_1 \right\rbrace ( H( f_1| \widetilde{M}_1) + H(f_2|\widetilde{M}_2)\\&+ 2 \max \lbrace 1, z_1, z_2 \rbrace (H_k(M_k|\widetilde{M}_k))).
\end{split}
\end{align*}}
Define $\widetilde{C}:= \text{min} \left\lbrace  \nu_{11} n_1 +  \nu_{12} n_2 , \nu_{22} n_2 + \nu_{21} n_1, \frac{\nu_{11} n_1}{z_1} + \nu_{12} n_2, \frac{\nu_{22} n_2}{z_2} + \nu_{21} n_1 \right\rbrace,$ then we can deduce an exponential decay with Gronwall's inequality
\begin{align*}
H(f_k|\widetilde{M}_k) &\leq \left(\sum_{k=1}^2 \left( H_k(f_k|\widetilde{M}_k)+ 2 \max \lbrace 1, z_1, z_2 \rbrace H_k(M_k|\widetilde{M}_k)\right) \right)\\& \leq e^{-\widetilde{C}t} \left(\sum_{k=1}^2 \left( H_k(f_k^0|\widetilde{M}_k^0)+ 2 \max \lbrace 1, z_1, z_2 \rbrace H_k(M_k^0|\widetilde{M}_k^0)\right) \right), \quad k=1,2.
\end{align*}
With the Ciszar-Kullback inequality (see \cite{Matthes}) we get
\begin{align*}
||f_k - M_k||_{L^1(dv)} & \leq 4 (H(f_k|\widetilde{M}_k))^{\frac{1}{2}} \\ & \leq 4 e^{-\frac{1}{2} Ct} [\left(\sum_{k=1}^2 \left( H_k(f_k^0|\widetilde{M}_k^0)+ 2 \max \lbrace 1, z_1, z_2 \rbrace H_k(M_k^0|\widetilde{M}_k^0)\right) \right)]^{\frac{1}{2}}.
\end{align*}
\end{proof}
Especially, we observe that  it is possible to choose $z_k \leq 1$ or $z_k \geq 1$ meaning that the model allows for slow and fast relaxation of the temperatures, and leads in the space-homogeneous case to a convergence of the distribution functions towards Maxwell distributions with a reasonable rate of convergence.

\section*{Acknowledgements}
We thank Christian Klingenberg for helpful revisions that improved this paper. We thank him and Gabriella Puppo for many discussions on polyatomic modelling and multi-species kinetics.



\end{document}